\documentclass[12pt,a4paper]{amsart}
\pdfoutput=1

\usepackage[utf8]{inputenc}
\usepackage[babel]{microtype}
\usepackage[T1]{fontenc}
\usepackage{lmodern}
\usepackage[british]{babel}
\usepackage[dvipsnames]{xcolor}
\usepackage{amsmath,amssymb,amsthm,mathtools}
\usepackage{tikz}
\usetikzlibrary{shapes.geometric}
\mathtoolsset{centercolon}
\usepackage{float}
\usepackage{caption,subcaption}

\usepackage{hyperref}
\hypersetup{
	colorlinks,
    linkcolor={red!60!black},
    citecolor={green!60!black},
    urlcolor={blue!60!black},
}
\usepackage[abbrev, msc-links, alphabetic, nobysame]{amsrefs}

\usepackage{geometry}
\usepackage[nameinlink]{cleveref}

\usepackage{enumitem,xspace}

\newtheorem{thm}{Theorem}[section]
\newtheorem{cor}[thm]{Corollary}
\newtheorem{lem}[thm]{Lemma}
\newtheorem{prop}[thm]{Proposition}
\newtheorem{claim}[thm]{Claim}

\theoremstyle{definition}

\newtheorem{obs}[thm]{Observation}

\newtheorem{quest}[thm]{Question}
\newtheorem{conj}[thm]{Conjecture}

\newcommand{\mcalP}{\mathcal{P}}
\newcommand{\mcalQ}{\mathcal{Q}}
\newcommand{\mcalG}{\mathcal{G}}

\newcommand{\braces}[1]{\left\lbrace #1 \right\rbrace}

\newcommand{\entr}[1]{\mathsf{in}_{#1}}

\crefname{section}{section}{sections}
\crefformat{section}{#2Section~#1#3}
\Crefformat{section}{#2Section~#1#3}

\crefname{subsection}{subsection}{subsections}
\AtBeginDocument{\crefformat{subsection}{#2Subsection~#1#3}}
\Crefformat{subsection}{#2Subsection~#1#3}

\crefname{lem}{lemma}{lemmata}
\crefformat{lem}{#2Lemma~#1#3}
\Crefformat{lem}{#2Lemma~#1#3}

\crefname{rem}{remark}{remarks}
\crefformat{rem}{#2Remark~#1#3}
\Crefformat{rem}{#2Remark~#1#3}

\crefname{thm}{theorem}{theorems}
\crefformat{thm}{#2Theorem~#1#3}
\Crefformat{thm}{#2Theorem~#1#3}

\crefname{cor}{corollary}{corollaries}
\crefformat{cor}{#2Corollary~#1#3}
\Crefformat{cor}{#2Corollary~#1#3}

\crefname{figure}{figure}{figures}
\crefformat{figure}{#2Figure~#1#3}
\Crefformat{figure}{#2Figure~#1#3}

\crefname{claim}{claim}{claims}
\crefformat{claim}{#2Claim~#1#3}
\Crefformat{claim}{#2Claim~#1#3}

\crefname{prop}{proposition}{propositions}
\crefformat{prop}{#2Proposition~#1#3}
\Crefformat{prop}{#2Proposition~#1#3}

\crefname{obs}{observation}{observations}
\crefformat{obs}{#2Observation~#1#3}
\Crefformat{obs}{#2Observation~#1#3}

\crefname{quest}{question}{questions}
\crefformat{quest}{#2Question~#1#3}
\Crefformat{quest}{#2Question~#1#3}

\title{Large vertex-flames in uncountable digraphs}

\author{Florian Gut}
\address{Florian Gut,
University of Hamburg, Department of Mathematics, Bundesstra{\ss}e 55 (Geomatikum), 
20146 Hamburg, Germany}
\email{florian.gut@uni-hamburg.de}
\author{Attila Jo\'{o}}
\thanks{The second author would like to thank the generous support of the Alexander 
von Humboldt Foundation and NKFIH 
OTKA-129211}
\address{Attila Jo\'{o},
University of Hamburg, Department of Mathematics, Bundesstra{\ss}e 55 (Geomatikum), 20146 Hamburg, Germany}
\email{attila.joo@uni-hamburg.de}
\address{Attila Jo\'{o},
Alfr\'{e}d R\'{e}nyi Institute of Mathematics, Set theory and general topology research division, 13-15 Re\'{a}ltanoda St., 
Budapest, Hungary}
\email{jooattila@renyi.hu}

\keywords{vertex-flame, connectivity, infinite Menger's theorem}
\subjclass[2020]{Primary: 05C63, 05C20, 05C40 Secondary: 05C38 } 
\begin{document}

\begin{abstract}
    The study of minimal subgraphs witnessing a connectivity property is an important field in graph theory.
    The foundation for large flames has been laid by Lov\'asz:
    Let $ D=(V,E) $ be a finite digraph and let $ r\in V $.
    The local connectivity $ \kappa_D(r,v) $ from $ r $ to $ v $ is defined to be the maximal number of internally disjoint $ r\rightarrow v $ paths in $ D $.
    A spanning subdigraph $ L $ of $ D $ with $ \kappa_L(r,v)=\kappa_D(r,v) $ for every $ v\in V-r $ must have at least $ \sum_{v\in V-r}\kappa_D(r,v) $ edges. 
    Lovász proved that, maybe surprisingly, this lower bound is sharp for every finite digraph.
    
    The optimality of an $ L $ satisfying the min-max criteria from Lov\'asz' theorem may instead also be captured by the following structural characterisation:
    For every $ v\in V-r $ there is a system $ \mathcal{P}_v $ of internally disjoint $ r\rightarrow v $ paths in $ L $ covering all the incoming edges of $ v $ in $ L $ such that one can choose from each $ P\in \mathcal{P}_v $ either an edge or an internal vertex in such a way that the resulting set meets every $ r\rightarrow v $ path of $ D $.
    The positive result for countably infinite digraphs based on this structural infinite generalisation were obtained by the second author. 
    
    In this paper we extend this to digraphs of size $ \aleph_1 $ which requires significantly more complex techniques.
    Despite solving yet the smallest uncountable case, the complete understanding of the concept and potentially a proof for arbitrary cardinality still seems to be far.
\end{abstract}

\maketitle

\section{Introduction}

The starting point of our investigation is the following theorem of Lovász.

\begin{thm}[Lovász, consequence of \cite{lovasz1973connectivity}*{Theorem 2} ]\label{t: LovFlame}
    Let $ D $ be a digraph with $ r\in V(D) $.
    Then there is a spanning subdigraph $ L $ of $ D $ in which for every $ v\in V(D)-r $ the following three quantities are equal:
    the local connectivities $ \kappa_D(r,v)$ and $ \kappa_L(r,v) $, and the in-degree of $ v $ in $ L $.
\end{thm}

Let us call a spanning subdigraph $ L $ of a finite `$ r $-rooted' 
digraph $D=(V,E)$ large (w.r.t.~$ D $) if $ L $ preserves all the local connectivities from the root, i.e.~$ \kappa_D(r, v) = \kappa_L(r, v) $ for every $ v\in V-r $.
Furthermore, a finite $ r $-rooted digraph $D=(V,E)$ is defined to be a vertex-flame if $ \kappa_D(r, v) = \left\vert \mathsf{in}_D(v)\right\vert $ for every $ v\in V-r $, where $ \mathsf{in}_D(v) $ is the set of incoming edges of $ v $. 
Using this terminology, Lovász' theorem says that every finite rooted digraph admits a large vertex-flame.
It was shown by Calvillo Vives in \cite{flameVives}, that in every finite $ r $-rooted digraph $ D $ every vertex-flame subgraph (with respect to root $ r $) be extended to a large vertex-flame of $ D $.
This was further generalised by the second author by proving that the edge sets of the vertex-flame subdigraphs of a finite rooted digraph $ D=(V,E) $ are the feasible sets of a greedoid on $ E $ whose bases are exactly the large vertex-flames in Lovász' theorem (see \cite{jooGreedoidFlame2021}*{Theorem 1.2}).

    There are many results in infinite graph theory that were first proved only for finite graphs and a deeper understanding of the underlying concept and more complex arguments were necessary to generalise them to infinite ones.
    Sometimes even the appropriate formulation of the problem for infinite graphs is already non-trivial because the equivalent forms in the finite case could be no more equivalent in general.
    For example it is well-known and easy to prove that the edge set of a finite graph can be partitioned into cycles if and only if there is no vertex with odd degree.
    The condition can be rephrased as the non-existence of odd cuts.
    A deep theorem of Nash-Williams \cite{nash1960decomposition}* {p.~235 Theorem 3} says that the reformulated condition is actually sufficient to partition the edges of a graph of any size into cycles, whereas the original condition is insufficient which is for example witnessed by a two-way infinite path.
    Other classical results fail at some cardinalities;
    for example every countable graph admits a normal spanning tree but an uncountable complete graph does not.

    For one of the most influential theorems in infinite graph theory the necessity of choosing the ``right'' formulation was also true.
    The result in question is the generalisation of Menger's theorem for arbitrary graphs which will play an important role in the main result of the paper.
    Erdős observed that the maximal size of a system of pairwise disjoint paths in a graph between two prescribed vertex sets and the minimal size of a vertex set meeting all the paths between these two sets is the same regardless of the size of the graph.
    He realised that considering this min-max formulation of Menger's theorem does not lead to a really strong infinite generalisation.
    Indeed, choosing the path-system inclusion-wise maximal and taking all the vertices of these paths as a separator is suitable whenever the path-system in question is infinite, although this separator is clearly way too ``big'' in a structural sense.
    Erdős conjectured the following structural infinite generalisation of Menger's theorem (it was known as the Erdős-Menger conjecture) which was eventually proved after several partial results by Aharoni and Berger:
\begin{thm}[Aharoni \& Berger \cite{aharoni2009menger}*{Theorem 1.6}]\label{t: inf Menger}
    For every digraph $ D $ and $ X,Y\subseteq V(D) $, there is a system $ \mathcal{P} $ of pairwise disjoint $ X \rightarrow Y $ paths in $ D $ such that one can choose exactly one vertex from each path in $ \mathcal{P} $ in such a way that the resulting vertex set $ S $ separates $ Y $ from $ X $ in $ D $. 
\end{thm}

As in the case of the Erdős-Menger conjecture, quantities are not appropriate to obtain the right infinite generalisation of \cref{t: LovFlame}, thus we need to look at the structural properties of $ L $.
We extend the definition of vertex-flames for rooted digraphs of any size by demanding for every $ v\neq r $ the existence of internally disjoint $ r\rightarrow v $ paths covering \emph{all} the incoming edges of $ v $ instead of just the equality of the in-degree of $ v $ and the local connectivity from $ r $ to $ v $.\footnote{One can define edge-flames by considering edge-disjoint paths.}
The condition $ \kappa_D(v)=\kappa_L(v) $ translates to the existence of a maximal-sized internally disjoint $ r\rightarrow v $ path-system $ \mathcal{P} $ of $ D $ that lives in $ L $.
We strengthen this by asking $ \mathcal{P} $ to be ``big'' not just cardinality-wise but in a structural Erdős-Menger way.
Namely, we demand the existence of a separation of $ v $ from $ r $ in $ D $ that can be obtained by choosing exactly one edge or one internal vertex from each path in $ \mathcal{P} $.
(The existence of such a $ \mathcal{P} $ is equivalent with \cref{t: inf Menger}).
A spanning subdigraph is called large if there is such a $ \mathcal{P} $ for every $ v\neq r $.
We will see that in a large vertex-flame for each $ v $ the path-system witnessing largeness and the path-system covering the incoming edges of $ v $ can be chosen to be the same (see the promised $ \mathcal{P}_v $ in the abstract).

The existence of large vertex-flames in the sense above in countable rooted digraphs was shown by the second author \cite{joo2019vertex}*{Theorem 1.2}.
He proved later with Erde and Gollin the strengthened of this result stating that every vertex-flames of a countable rooted digraph can be extended into a large one \cite{erde2020enlarging}*{Theorem 1.3}.
The main result of the paper is leaving countable digraphs behind and handle the smallest uncountable case:
\begin{thm}\label{t: main result}
    Every rooted digraph of size at most $ \aleph_1 $ admits a large vertex-flame.
\end{thm}
As in the case of the infinite version of Menger's theorem, \cref{t: inf Menger}, the construction and the necessary arguments get significantly more complex when the digraph in question is uncountable.
Although several of our tools can be used to approach the problem for arbitrary large digraphs, our proof relies strongly on the fact that there is an enumeration of the vertex set in which the proper initial segments are countable.
We expect that \cref{t: main result} remains true without any size restriction on the digraph but we feel that, despite solving the smallest uncountable case, a complete understanding of the problem is still far.
\begin{conj}
    Every rooted digraph admits a large vertex-flame.
\end{conj}

The following edge-variant of the problem is wide open even in the countable case, but known for finite digraphs even in a fractional variant with edge-capacities and ``flow-connectivity'' (see \cite{jooGreedoidFlame2021}*{Theorem 4.1}).

\begin{quest}
    Let $ D $ be a countable digraph with $ r\in V(D) $.
    Is it always possible to find a spanning subdigraph $ L $ of $ D $ such that for every $ v\in V(D)-r $ there is a system $ \mathcal{P}_v $ of \emph{edge-disjoint} $ r\rightarrow v $ paths in $ L $ covering all the incoming edges of $ v $ in $ L $ such that one can choose exactly one \emph{edge} from each $ P\in \mathcal{P}_v $ in such a way the resulting edge set is an $ rv $-cut in $ D $?
\end{quest}
\section{Definitions and notation}

\subsection{Digraphs}
All the \emph{digraphs} $ D $ in the paper are simple and have no incoming edges to their `root vertex' $ r $ whenever they have such a root. 
We denote the set of incoming edges of a vertex set $ X $ by $\boldsymbol{\mathsf{in}_{D} (X)} $ and 
$\boldsymbol{\mathsf{out}_{D} (X)}$ stands for the set of the outgoing edges.
For the in-neighbours of $ X $ (i.e.~the tails of the edges in $ \mathsf{in}_{D} (X) $) we write $ \boldsymbol{N^{-}_{D}(X)} $ 
and the out-neighbours $ \boldsymbol{N^{+}_{D}(X)} $ defined analogously.
The subdigraph induced by a vertex set $ U $ is $ \boldsymbol{D[U]} $ and $ \boldsymbol{H\subseteq D} $ expresses that $ H $ 
is a subdigraph of $ D $.
We define $ \boldsymbol{D_0\cap D_1}:=(V_0\cap V_1, E_0\cap E_1) $ if $ D_i=(V_i, E_i) $ are digraphs. 

\subsection{Paths}
All the paths in the paper are finite and directed, repetition of vertices is not allowed.
A path is \emph{trivial} if it consists of a single vertex.
An $ \boldsymbol{X\rightarrow Y} $ path is a path whose first and last vertex lie in the vertex set $ X $ and $ Y $ respectively 
and has has no internal vertex in $ X\cup Y $.
For paths $ P $ and $ Q $ with $ v\in V(P)\cap V(Q) $, let $ \boldsymbol{PvQ} $ be the digraph consisting of the initial segment 
of $ P $ up to $ v $ and the terminal segment of $ Q $ from $ v $.
A \emph{path-system} (i.e.~set of paths) $ \mathcal{P} $ is \emph{disjoint} if the paths in it are pairwise vertex-disjoint. 
We define \emph{internally disjoint} similarly except that the first and last vertices are allowed to be shared.
We denote the united vertex set and edge set of the paths in $ \mathcal{P} $ by $ \boldsymbol{V(\mathcal{P})} $ and by $ \boldsymbol{E(\mathcal{P})} $ respectively.
Let us write $ \boldsymbol{V^{-}(\mathcal{P})} $ and $ \boldsymbol{V^{+}(\mathcal{P})} $ for the respective set of the first and last vertices of the paths in $ \mathcal{P} $. 
We define $ \boldsymbol{E^{-}(\mathcal{P})} $ and $ \boldsymbol{E^{+}(\mathcal{P})} $ similarly with edges but only for path-systems without trivial paths. We write simply $ \boldsymbol{\mathsf{in}_{\mathcal{P}}(v) }$ for the set of the incoming edges of $ v $ in the digraph  $ (\bigcup_{P\in \mathcal{P}}V(P), \bigcup_{P\in \mathcal{P}}E(P)) $.
A $ v $\emph{-fan} is a system of paths sharing only their initial vertex $ v $.
A $v $\emph{-infan} is what we obtain by reversing the edges of a $ v $-fan.
A set $ X\subseteq V-v $ is \emph{linked from} $ v $ in $ D $ if there is a $ v $-fan $ \mathcal{P} $ in $ D $ with $ V^{+}(\mathcal{P})=X $.
Similarly, $ X $ is \emph{linked to} $ v $ if there is a $ v $-infan $ \mathcal{P} $ with $ V^{-}(\mathcal{P})=X $. 

\subsection{Vertex-flames}
Let $ V $ be some fixed vertex set with a prescribed `root vertex' $ r\in V $.
For a(n $ r $-rooted) digraph $ D $, $ v\in V-r $ and an arbitrary set $ I $ we write $\boldsymbol{ D\upharpoonright_v I} $ for the subdigraph we obtain from $ D $ by deleting those incoming edges of $ v $ that are not in $ I+rv $.
For $v \in V-r$ we denote by $\boldsymbol{\mathcal{G}_D(v)}$ the set of those $ I \subseteq \mathsf{in}_{D}(v)$ for which there exists an internally disjoint $r \rightarrow v$ path-system $\mathcal{P}$ in $ D $ with $ E^+ (\mathcal{P}) = I$.
We say that $ D $ has the \emph{vertex-flame property} at $ v\in V-r $ if $ \mathsf{in}_D(v)\in \mathcal{G}_D(v) $ and we call $ D $ a \emph{vertex-flame} if it has the vertex-flame property at every $ v\in V-r $.
The \emph{quasi-vertex-flame property} at $ v $ means that all the finite subsets of $ \mathsf{in}_D(v) $ are in $ \mathcal{G}_D(v) $ and $ D $ is a \emph{quasi-vertex-flame} if it has the quasi-vertex-flame property at every $ v\in V-r $.

\subsection{\texorpdfstring{Erd\H{o}s}{Erdos}-Menger separations and path-systems}\label{subsection EM sep path}
For $ S\subseteq V-r-v $ let $ \boldsymbol{\mathfrak{P}_D(v, S)} $ be the set of those internally disjoint $ r\rightarrow v $ path-systems $ \mathcal{P} $ in $ D $ that are \emph{orthogonal} to $ S $, i.e.~for which $ S $ can be obtained by choosing exactly one internal vertex from each $ P\in \mathcal{P} $ (observe that a path consisting of a single edge cannot be in $\mathfrak{P}_D(v, S) $).
For $ v\in V-r $, we define $ \boldsymbol{\mathfrak{S}_D(v)} $ to be the set of \emph{Erdős-Menger separations} between $ r $ and $ v $, i.e.~the set of those $ S\subseteq V-r-v $ that separate $ r $ from $ v $ in $ D-rv $ (separation means that every $ r\rightarrow v $ path in $ D-rv $ meets $ S $) and for which $ \mathfrak{P}_{D}(v,S)\neq \varnothing $.
We call $ \boldsymbol{\mathfrak{P}_D(v)}:=\bigcup\{\mathfrak{P}_D(v, S):\, S\in \mathfrak{S}_D(v) \} $ the set of the \emph{Erdős-Menger paths-systems}.
Note that the infinite version of Menger's theorem, \cref{t: inf Menger}, applied to $X=N^{+}_{D-rv}(r) $ and $ Y=N^{-}_{D-rv}(v) $ in $ D-rv $ ensures that $ \mathfrak{S}_D(v)\neq \varnothing $ and therefore $ \mathfrak{P}_D(v)\neq \varnothing $. 
Observe that an $ S\in \mathfrak{S}_{D}(v) $ is always a minimal $ rv $-separation in $ D-rv $ since for every $ s\in S $ each $ \mathcal{P}\in \mathfrak{P}_{D}(v, S) $ contains some $ r\rightarrow v $ path $ P_s $ that meets $ S $ only at $ s $.
One can show (see \cite{joo2019complete}*{Theorem 3.5}) that $ \mathfrak{S}_D(v) $ is a complete lattice with respect to the partial order in which $ S\leq T $ if $ S $ separates $ T $ from $ r $ (equivalently $ T $ separates $ S $ from $ v $) in $ D-rv $.
We denote the smallest and the largest element of $ \mathfrak{S}_D(v) $ by $ \boldsymbol{S_{D,v}} $ and $ 
\boldsymbol{T_{D,v}} $, respectively.

\subsection{Large spanning subdigraphs}
A system $\mathcal{P}$ of internally disjoint $ r\rightarrow v $ paths in $ D $ is called \emph{strongly maximal} w.r.t.~$ D $ if for every internally disjoint $ r\rightarrow v $ path-system $ \mathcal{Q} $ we have $ \left|\mathcal{Q}\setminus \mathcal{P}\right|\leq \left|\mathcal{P}\setminus \mathcal{Q}\right| $.
In a finite $ D $ strongly maximal simply means `maximal-sized' but in general digraphs it is a stronger assumption than cardinality-wise maximality.
It is not too hard to prove that the set of the strongly maximal internally disjoint $ r\rightarrow v $ path-systems in $ D $ is exactly $ \mathfrak{P}_D(v) $ if $ rv\notin E(D) $ and the extensions of the elements of $ \mathfrak{P}_D(v) $ with the single-edge path $ rv $ if $ rv\in E(D) $ (see for example \cite{joo2019vertex}*{Proposition 3.4}).
For a fixed $ D $ and $v \in V(D)-r$ we call a spanning subdigraph $ L $ of $ D $ $ v $\emph{-large} w.r.t.~$ D $ if there is a strongly maximal internally disjoint $ r\rightarrow v $ path-system of $ D $ that lies in $ L $, moreover, $ L $ is $ D $-\emph{large} (shortly large if $ D $ is fixed) if it is $ v $-large w.r.t. $ D $ for every $ v\in V-r $.
For a finite $ D $ the largeness of $ L\subseteq D $ is equivalent with the preservation of the local connectivities from the root, i.e.~with $ \kappa_{L}(v)=\kappa_{D}(v) $ for every $ v\in V-r $ but it has a stronger structural meaning for general digraphs.
Largeness of $ L $ can be rephrased as: $ \mathfrak{P}_{D}(v)\cap \mathfrak{P}_{L}(v)\neq \varnothing $ (equivalently $ \mathfrak{S}_{D}(v)\cap \mathfrak{S}_{L}(v)\neq \varnothing $) for every $ v\in V-r $ and $ \mathsf{out}_D(r)\subseteq L $.

\section{Preliminaries and preparations}
\subsection{Elementary submodels}
Elementary submodels are defined for first order structures in logic but for simplicity let us talk only about the special case we use.
An elementary submodel of a set $ A $ is an $ M\subseteq A $ such that for every first order formula $ \varphi(x_1,\dots, x_n) $ in the language of set theory (with free variables $ x_1,\dots, x_n $) and for every $ a_1,\dots, a_n\in M $, the statement $ \varphi(a_1,\dots, a_n) $ is true in the first order structure $ (A, \in|_{A\times A}) $ if and only if it is true in $ (M, \in|_{M\times M}) $.
Elementary submodels provide a powerful method in topology, infinite combinatorics and in other fields to cut up uncountable structures into smaller ``well-behaved'' pieces.
In these applications $ A $ usually consists of the sets whose transitive closure is of cardinality less than $ \lambda $ (denoted by $ H(\lambda) $), where $ \lambda $ is chosen in such a way that $ H(\lambda) $ contains all the sets that are relevant in the proof.
By \emph{elementary submodel} we always mean an elementary submodel of $ H(\lambda) $ for a large enough $ \lambda $.
For a detailed introduction for elementary submodel techniques and their applications in infinite combinatorics we refer to \cite{soukup2011elementary}.
For an elementary submodel $ M $ and digraph $ D=(V,E) $ we let $ \boldsymbol{D\cap M}:=(V\cap M, E\cap M) $.

\subsection{A reduction to quasi-vertex-flames}
\begin{lem}[\cite{joo2019vertex}*{Lemma 2.1}]\label{large quasi flame}
    For every rooted digraph $ D $, there is a quasi-vertex-flame $ F\subseteq D $ such that whenever an $ L\subseteq F $ is $ F $-vertex-large it is $ D $-vertex-large as well. 
\end{lem}

\begin{cor}\label{cor: wlog D is quasi-flame}
    One may assume without loss of generality in the proof of \cref{t: main result} that $ D $ is a quasi-vertex-flame.
\end{cor}

\subsection{Linkability of finite sets from \texorpdfstring{$ \boldsymbol{r} $}{r}}
\begin{lem}[\cite{joo2019vertex}*{Claim 3.14}]\label{l: linked from r U}
    If a finite $ U\subseteq V-r $ is linked from $ r $ in $ D $ and $ L $ is large, then $ U $ is linked from $ r $ in $ L $ as well.
\end{lem}

\begin{cor}[\cite{joo2019vertex}*{Lemma 2.3}]\label{cor: large for a quasi is quasi}
If $ D $ is a quasi-vertex-flame and $ L $ is large, then $ L $ is also a quasi-vertex-flame.
\end{cor}

\subsection{Variants of Pym's theorem}
\begin{thm}[Pym's theorem \cite{pym1969linking}*{The Linkage Theorem}]
    Let $ X,Y\subseteq V $, furthermore, let $ \mathcal{P} $ and $ \mathcal{Q} $ be disjoint systems of $ X\rightarrow Y $ paths.
    Then there is a system $ \mathcal{R} $ of disjoint $ X\rightarrow Y $ paths such that $ V^{-}(\mathcal{R}) \supseteq V^{-}(\mathcal{P})$ and $ V^{+}(\mathcal{R}) \supseteq V^{+}(\mathcal{Q})$, moreover, each $ R\in \mathcal{R} $ is either in $ \mathcal{P}\cup \mathcal{Q} $ or there are $ P\in \mathcal{P},\ Q\in \mathcal{Q} $ and $ v_R\in V(P)\cap V(Q) $ such that $ R=Pv_R Q $. 
\end{thm}

\begin{cor}\label{cor: Pym}
    Suppose that $ \mathcal{P} $ links $ S$ to $ v $ and $ \mathcal{Q} $ is a $ v $-infan with $ V(\mathcal{Q})\cap S=V^{-}(\mathcal{Q}) $.
    Then there is a $ v $-infan $ \mathcal{R} $ with $ V^{-}(\mathcal{R})=S $ covering $ E^{+}( \mathcal{Q}) $, furthermore, each $ R\in \mathcal{R} $ is either in $ \mathcal{P}\cup \mathcal{Q} $ or there are $ P\in \mathcal{P},\ Q\in \mathcal{Q} $ and $ v_R\in V(P)\cap V(Q) $ such that $ R=Pv_R Q $. 
\end{cor}
We need one more version of the theorem in which $ r\in S $ and more than one path in $ \mathcal{P} $ and in $ \mathcal{R} $ may start in $ r $.
This variant can be reduced to \cref{cor: Pym} by splitting $ r $ into a vertex set $V_r:= \{ r_e: e\in \mathsf{out}_D(r) \} $ where $ r_e $ inherits the single outgoing edge $ e $ of $ r $.
\begin{cor}\label{cor: Pym+}
    Suppose that $ \mathcal{P} $ is a system of $ S\rightarrow v $ paths with $ v\notin S $ and such that $ V(P_0)\cap V(P_1)-v\subseteq \{ r \} $ for every $ P_0\neq P_1 $ from $ \mathcal{P} $ and suppose $ \mathcal{Q} $ is a $ v $-infan with $ V(\mathcal{Q})\cap S=V^{-}(\mathcal{Q}) $.
    Then there is a system $ \mathcal{R} $ of $ S \rightarrow v $ paths with $V(R_0)\cap V(R_1)-v\subseteq\{ r \} $ for every $ R_0\neq R_1 $ from $ \mathcal{R} $ covering $ V^{-}(\mathcal{P})\cup E^{+}( \mathcal{Q}) $, furthermore, each $ R\in \mathcal{R} $ is either in $ \mathcal{P}\cup \mathcal{Q} $ or there are $ P\in \mathcal{P},\ Q\in \mathcal{Q} $ and $ v_R\in V(P)\cap V(Q) $ such that $ R=Pv_R Q $. 
\end{cor}

\begin{cor}\label{cor: EMpathsys covers I}
    Let $ S\in \mathfrak{S}_D(v) $ and let $ I\in \mathcal{G}_{D}(v) $.
    Then there is an $\mathcal{R}\in \mathfrak{P}_{D}(v,S) $ with $I-rv\subseteq E^{+}(\mathcal{R}) $.
\end{cor}
\begin{proof}
   Let $ \mathcal{P}\in \mathfrak{P}_{D}(v,S) $ and let $ \mathcal{Q} $ be a witness for $ I\in \mathcal{G}_{D}(v) $.
   We define $ \mathcal{Q}' $ to be set the of terminal segments of the paths in $ \mathcal{Q} $ from the last intersection with $ S $.
    Let $ \mathcal{P}' $ consist of the terminal segments of the paths in $ \mathcal{P} $ from $ S $.
    We apply \cref{cor: Pym} with $ \mathcal{P}' $ and $ \mathcal{Q}' $ and extend the paths in the resulting $ \mathcal{R}' $ backwards to $ r $ by the initial segments of the paths in $ \mathcal{P} $ up to $ S $ to obtain $ \mathcal{R} $.
\end{proof}

\subsection{Preservation of the vertex-flame property}
\begin{lem}[\cite{erde2020enlarging}*{Lemma 4.10}]
    Suppose that~${I \in \mathcal{G}_D(w)}$ such that~${(I+f) \in \mathcal{G}_D(w)}$ for every~${f \in\mathsf{in}_D(w) \setminus I}$. 
    Assume that there is a~${uv \in E(D)}$ with~${u \neq r, v \neq w}$ in such a way that~${I \notin \mathcal{G}_{D-uv}(w)}$. 
    Then there exists a set~${S \subseteq V - r}$ with $ v\in S $ which is linked from $ r $ by a path-system~$\mathcal{P}$, such that~$S$ separates $ N^{-}_{D}(v)-u $ from~$r$. In particular,~$uv$ is the last edge of some~${P_{uv} \in \mathcal{P}}$.
\end{lem}

 We are interested only in the special case where $ I=\mathsf{in}_D(w) $:
\begin{cor}\label{cor: G-quasi add one}
    Suppose that~${\mathsf{in}_D(w)\in \mathcal{G}_D(w)}$ and there is a~${uv \in E(D)}$ with~${u \neq r, v \neq w}$ for which~${\mathsf{in}_D(w) \notin \mathcal{G}_{D-uv}(w)}$. 
    Then there exists a set~${S \subseteq V - r}$ with $ v\in S $ which is linked from $ r $ by a path-system~$\mathcal{P}$, such that~$S$ separates $ N^{-}_{D}(v)-u $ from~$r$. 
    In particular,~$uv$ is the last edge of some~${P_{u,v} \in \mathcal{P}}$.
\end{cor}

A digraph $ D $ has the $ G $\emph{-quasi-vertex-flame property} for some  $ G\subseteq D $ at $ v\in V-r $ if $ \mathcal{G}_{D}(v) $ contains every $ I\subseteq \mathsf{in}_D(v) $ for which $ I\setminus \mathsf{in}_{G}(v) $ is finite. 
(For a single $ v $ only the edges $ \mathsf{in}_G(v) $ are relevant for the $ G $-quasi-vertex-flame property at $ v $.)
We need a statement that closely resembles \cite{erde2020enlarging}*{Lemma 2.10}.
In fact, even though the statement of \cref{lem:G_quasi_flame} appears stronger on first glance since we drop the condition that $D$ should have the $G$-quasi-flame property at every vertex, \cref{lem:G_quasi_flame} is obtained with the proof in \cite{erde2020enlarging}*{Lemma 2.10}.
\begin{lem}[\cite{erde2020enlarging}*{Lemma 2.10}]\label{lem:G_quasi_flame}
    Assume that $D=(V,E)$ is a countable $ r $-rooted digraph, $ v\in V-r $ and $ G \subseteq D $. 
    Then there is an ${I^* \in \mathcal{G}_D(v)}$ such that~${D \upharpoonright_v I^*}$ has the $ G $-quasi-vertex-flame property for every $ u\in V-r $ for which $ D $ has this property.
\end{lem}

On the one hand, we are interested only in cases where $ G $ has a certain special form.
On the other hand we want to weaken the assumption that $ D $ is countable.
The following variant will be suitable for our purpose: 
\begin{cor}\label{cor: I G-qusi}
    Assume that $D=(V,E)$ is an $ r $-rooted digraph, ${v \in V-r}$ and $W \subseteq V-v-r $ is a countable set such that $ D $ has the vertex-flame property at every $ w\in W $.
    Then there is an ${I^* \in \mathcal{G}_D(v)}$ such that ${D \upharpoonright_v I^*}$ also has the vertex-flame property for every $ w\in W $.
\end{cor}

\begin{proof}
    Clearly, $D$ has the $G$-quasi-vertex-flame property for each $ w\in W $ for $ G :=(V,\{ xw \in E(D) \colon w \in W \}) $ by assumption.
    Let $v$ be the vertex given in the corollary and let $M$ be a countable elementary submodel with $v, r, G, D, W \in M$. 
    Note that $ W\subseteq M $ because $ W $ is countable.
    First we show that $D \cap M $ has the $G\cap M$-quasi-vertex-flame property for each $ w\in W $.
    Let $w \in W$ be arbitrary and let $I \subseteq \entr{D \cap M} (w)$ such that $ I\setminus \entr{G \cap M} (w) $ is finite.
    As $D$ has the $G$-quasi-vertex-flame property at $ w $, it follows that $I^\prime := I \cup \entr{G} (w) \in \mcalG_D(w)$ and thus there is a path system $\mcalP_{I'} $ witnessing this.
    Since $ G, w \in M $ and also $I\setminus \entr{G \cap M} (w)\in M $ (because $ I\setminus \entr{G \cap M} (w) $ is a finite subset of $ M $), we know that $ I^\prime\in M $ and therefore $ \mcalP_{I'} $ can be 	chosen to be an element of $ M$.
    We claim that $\mcalP_{I'}\cap M $ witnesses $ I\in \mathcal{G}_{D\cap M}(w) $. 
    Indeed, for $ e\in I'\cap M=I $ the unique path $ P_e\in \mcalP_{I'} $ through $ e $ is definable from $ e $ and $ \mcalP_{I'} $, thus $ P_e\in M $.
    Since $ P_e $ is finite we obtain $ P_e\subseteq M $ and therefore $ P_e $ is a path of $ D\cap M $.
    Furthermore, if the last edge of some $ P\in \mathcal{P}_{I'} $ is not in $ M $ then $ P\notin M $.
    Thus $ \mathcal{P}_{I'}\cap M $ consists of those paths in $ \mathcal{P}_{I'} $ whose last edge is in $ I $ and all these paths lie completely in $ D\cap M $, hence $ I\in \mathcal{G}_{D\cap M}(w) $ holds.
    
    We may now apply \cref{lem:G_quasi_flame} to the countable digraph $D\cap M$ with $G \cap M$ and $v$.
    This yields a set $I^\ast \in \mcalG_{D\cap M}(v)$ such that $D\cap M \upharpoonright_{v} I^\ast$ has the $G \cap M$-quasi-vertex-flame property ate every $ w\in W $.
    We shall show that $D \upharpoonright_{v} I^\ast$ has the vertex-flame-property at every $w \in W$. 
    Let $ w\in W $ be fixed.
    Let us pick a $ \mathcal{P} \in M $ witnessing the vertex-flame property of $D\cap M \upharpoonright_{v} I^\ast$ at $ w $ and a $ \mathcal{Q} \in M $ showing the vertex-flame property of $D$ at $ w $.
    We claim that $ \mathcal{R}:=\mathcal{P}\cup (\mathcal{Q}\setminus M) $ witnesses the vertex-flame property of $D\upharpoonright_v I^{*}$ at $ w $.
    Indeed, for any path $Q \in\mathcal{Q}$ meeting a vertex $u \in (V-r-w) \cap M$ we have $Q\in M$ and therefore $ Q\subseteq D\cap M $, as $ Q $ is finite and definable in $M$ from $ w $ and $ \mathcal{Q} $.
    This shows that $ \mathcal{R} $ is an internally disjoint path-system.
    Furthermore, if $ e\in 	\mathsf{in}_{D}(w) $, then $ e $ is the last edge of a path in $ \mathcal{P} $ or in $ \mathcal{Q}\setminus M $ depending on if $ e\in M $.
    This finishes the proof of the corollary.
\end{proof}

\subsection{Preservation of largeness}
We introduce some terminology that we are going to use only locally to prove some lemmas applying previous results. 
For an $ X\subseteq V-r $, the \emph{entrance} of $ X $ with respect to $ D $ is 
\[
    \boldsymbol{\mathsf{ent}_D(X)}:= \{ v\in X \colon \exists uv\in E(D)\text{ with }u\notin X \},
\]
and $ \boldsymbol{\mathsf{int}_D(X)} $ stands for its \emph{interior} $ X\setminus \mathsf{ent}_D(X)$.
A set $ B\subseteq V-r $ is a $v $\emph{-bubble} with respect to $ D $ if there exists a $ v $-infan $ \mathcal{P}=\{ P_{u} \colon u\in \mathsf{ent}_D(B)-v \} $ in $ D[B] $ where $P_u $ starts at $ u $.
Let us denote the set of the $ v $-bubbles in $ D $ by $ \boldsymbol{\mathsf{bubb}_D(v)} $.
Clearly $ \{ v \}\in \mathsf{bubb}_D(v) $ since either the trivial path consisting of the single vertex $ v $ or the empty set is a witness for it depending on if $ v\in \mathsf{ent}_D(\{ v \}) $. 
\begin{lem}[Bubble uniting lemma, \cite{joo2019vertex}*{Lemma 3.5}]\label{l: bubble unite}
    Let $ \alpha $ be an ordinal number.
    Suppose that $ \left\langle B_\beta: \beta<\alpha \right\rangle $ is a sequence where $ B_\beta\in \mathsf{bubb}_D(v_\beta) $ for some $ v_{\beta}\in V-r $.
    Let us denote $ \bigcup_{\gamma<\beta}B_\gamma $ by $ \boldsymbol{B_{<\beta}} $.
    If for each $ \beta<\alpha $ either $ v_\beta=v_0 $ or $ v_\beta\in \mathsf{int}_D \left(B_{<\beta} \right) $, then $ B_{<\alpha}\in \mathsf{bubb}_D(v_0) $.
\end{lem}
Note that for an $ S\in \mathfrak{S}_D(v) $, the set $ \boldsymbol{B_{D,S,v}}$ of vertices that are separated from $ r $ by $ S $ in $ D-rv $ form a $ v $-bubble with $\mathsf{ent}_{D-rv}(B_{D,S,v})=S $ such that $ N^{-}_{D-rv}(v)\subseteq B_{D,S,v}$.

\begin{cor}\label{cor: bubb in-neigbh}
    There is a $ \subseteq $-largest $ v $-bubble $ \boldsymbol{B_{D,v}} $ in $ D $ for every $ v\in V-r $ and it contains $ N^{-}_{D-rv}(v) $.
\end{cor}
\begin{lem}[\cite{joo2019vertex}*{Lemma 3.10}]\label{l: largest EMsep is real}
    A spanning subdigraph $ L $ of $ D $ is large if and only if $ u\in B_{L,v} $ for every $ uv\in E(D)\setminus E(L)$. 
    Furthermore, if $ L $ is large and $ v\in V-r $, then $ \mathsf{ent}_{L-rv}(B_{L,v})= \mathsf{ent}_{D-rv}(B_{L,v})\in \mathfrak{S}_D(v) $.
\end{lem}
Note that this also shows ${S_{L,v}} = \mathsf{ent}_{D-rv}(B_{L,v})$.

\begin{cor}\cite{joo2019vertex}*{Lemma 2.2}\label{cor: presLare}
    Assume that $\mathsf{out}_D(r)\subseteq  L\subseteq D $ such that for every $ v\in V-r $ with $ \mathsf{in}_L(v)\subsetneq \mathsf{in}_D(v) $ there is a $ \mathcal{P}\in \mathfrak{P}_D(v) $ that lies in $ L $.
    Then $ L $ is large.
\end{cor}

We call an $ A\subseteq V-r $  \emph{anti-bubble} in $ D $ if $ \mathsf{ent}_{D}(A) $ is linked from $ r $ in $ D $.
Note that the family of anti-bubbles are closed under arbitrarily large intersection.
For $ S\in \mathfrak{S}_D(v) $ the set $ B_{D,S,v} $ is not just a $ v $-bubble but also an anti-bubble that contains $ \{ v \}\cup N_{D-rv}^{-}(v) $.
Moreover, if $ X $ is a $ v $-bubble and also an anti-bubble and contains $ \{ v \}\cup N_{D-rv}^{-}(v) $, then $ \mathsf{ent}_{D-rv}(X)\in \mathfrak{S}_D(v) $.
Let $\boldsymbol{A_{D,v}} $ be the intersection of all anti-bubbles in $ D $ containing $ \{ v \}\cup N_{D-rv}^{-}(v) $.
\begin{prop}\label{prop: A D v is a v-bubble}
    For every $ v\in V-r $, $ A_{D,v} $ is a $ v $-bubble in $ D $, furthermore, we have $ {T_{D,v}}=\mathsf{ent}_{D-rv}(A_{D,v})\in \mathfrak{S}_{D}(v) $.
\end{prop}
\begin{proof}
    We apply \cref{t: inf Menger} (Aharoni-Berger) in $ D[A_{D,v}] $ with $ \mathsf{ent}_{D-rv}(A_{D,v}) $ and $ N^{-}_{D}(v) $.
    If the resulting separation $ S $ is not $ \mathsf{ent}_{D-rv}(A_{D,v}) $ itself, then $ B_{D,S,v}\subsetneq A_{D,v} $ is an anti-bubble containing $ \{ v \}\cup N_{D-rv}^{-}(v) $, which contradicts the minimality of $ A_{D,v} $.
    Thus $ S=\mathsf{ent}_{D-rv}(A_{D,v}) $, and hence the path-system given by \cref{t: inf Menger} witnesses that $A_{D,v} $ is a $ v $-bubble.
    Since $A_{D,v} $ is an anti-bubble as well and contains $ \{ v \}\cup N_{D-rv}^{-}(v) $, we have $\mathsf{ent}_{D-rv}(A_{D,v})\in \mathfrak{S}_{D}(v) $.
    It now follows from the definition that $\mathsf{ent}_{D-rv}(A_{D,v}) = {T_{D,v}}$
\end{proof}

Before we proceed we need another \namecref{l: aug walk}.
One of the standard proofs of Menger's theorem is based on the so called Augmenting walk lemma.
For a given disjoint system $ \mathcal{P} $ of $ X\rightarrow Y $ paths it either provides a bigger such system or an $X$-$Y$-separation consisting of exactly one vertex from each of the paths in $ \mathcal{P} $.
The infinite generalisation of this lemma (see Lemmas 3.3.2 and 3.3.3 in \cite{MR3822066}) was an important tool in the proof of the infinite version of \cref{t: inf Menger}.
There are several variants of the \namecref{l: aug walk} depending on whether the paths are edge-disjoint or vertex-disjoint, whether we consider graphs or digraphs etc.~but the proofs of these variants are essentially the same.
We make use of the following variant:
\begin{lem}[Augmenting walk]\label{l: aug walk}
    Assume that $ D=(V,E) $ is a digraph, $ X\subseteq V $ and $ v\in V\setminus X $.
    Let $ \mathcal{P}$ be a $ v $-infan with $ V(\mathcal{P})\cap X= V^{-}(\mathcal{P}) $.
    Then there is either an $ S $ that separates $ v $ from $ X $ consisting of a unique $ v_P\in V(P)-v $ for every $P\in \mathcal{P} $ or there is a $ v $-infan $ \mathcal{Q} $ with $ V(\mathcal{Q})\cap X= V^{-}(\mathcal{Q}) $ such that $ \left|\mathcal{P}\setminus \mathcal{Q}\right|+1= \left|\mathcal{Q}\setminus \mathcal{P}\right|<\aleph_0$ and $ V^{-}(\mathcal{Q})\supseteq V^{-}(\mathcal{P}) $.
 \end{lem}
    We say that the augmentation is \emph{successful} if the second case occurs and we say that it is \emph{unsuccessful} otherwise.
\begin{lem}\label{l: largest EMsep}
    Assume that $ I\subseteq \mathsf{in}_D(v) $ such that $ T_{D,v} $ remains linked to $ v $ in $ D':=D \upharpoonright_v I $. 
    Then for every $ u\in V-r$ every $ S\in \mathfrak{S}_{D}(u)$ remains linked to $ u $ in $ D' $. 
\end{lem}
\begin{proof}
    Let $ u\in V-r-v $ and $S \in \mathfrak{S}(u)$ be given and let $ \mathcal{P} $ be a path-system that links $ S $ to $ u $ in $ D $.
    We may assume that there is some $ e\in E(D)\setminus E(D') $ such that there is a $ P_e\in \mathcal{P} $ through $ e $ since otherwise $ \mathcal{P} $ is a path-system in $ D' $ as well and we are done.
    
    We apply \cref{l: aug walk} in $ D' $ with $ S, u $ and $ \mathcal{P}-P_e $.
    If the augmentation is successful, the resulting path system witnesses that $S$ is linked to $u$ in $D'$ and we are done.
    Therefore suppose that the augmentation is unsuccessful.
    Then we can choose a unique $ v_P\in V(P)-u $ from each $ P\in \mathcal{P}-P_e $ such that the resulting $ S' $ separates $ u $ from $S $ in $ D' $.
    Then $ B_{D',S',u}$ is a $ u $-bubble in $ D' $ by definition.

    Next, we want to show that $ B:= A_{D,v} \cup B_{D',S',u} \in \mathsf{bubb}_{D'}(u)$ via \cref{l: bubble unite}.
    In order to do this we need to show $A_{D,v} \in \mathsf{bubb}_{D'}(v)$ and $v \in \mathsf{int}_{D'}(B_{D',S',u})$.
    For $A_{D,v} \in \mathsf{bubb}_{D'}(v)$, recall that $T_{D,v}$ remains linked to $v$ in $D'$ by assumption.
    Since we only deleted incoming edges of $v$ and $A_{D,v} \supseteq N_{D-rv}^{-}(v)$ by definition, $A_{D,v}$ is a $v$-bubble in $D'$.
    For $v \in \mathsf{int}_{D'}(B_{D',S',u})$ note that the terminal segment $vP_{e}u$ of $ P_e $ must lie in $ B_{D',S',u} $ since otherwise $ S' $ would not separate $ u $ from $S $ in $ D' $.
    This implies $ v\in B_{D',S',u} $ in particular.
    Since $ v\notin S' $ by construction but $\mathsf{ent}_{D'}(B_{D',S',u}) = S'$ by definition, we can conclude $ v\in \mathsf{int}_{D'}(B_{D',S',u}) $.
    Thus we really may apply \cref{l: aug walk} to $A_{D,v}$ and $B_{D',S',u}$ and obtain $ B \in \mathsf{bubb}_{D'}(u) $.

    For the construction of a path system witnessing that $S$ is linked to $u$ in $D'$, we show next that $B \subseteq B_{D,S,u}$ by showing this for $A_{D,v}$ and $B_{D',S',u}$ separately.
    For $A_{D,v} \subseteq B_{D,S,u}$, note that $ v $ is the head of $ e $, as we only deleted incoming edges at $v$.
    Since $u \neq v$, $ V^{-}(\mathcal{P})=\mathsf{ent}_{D}(B_{D,S, u})-u $ and the paths in $ \mathcal{P} $ are pairwise disjoint, this implies $ v\in \mathsf{int}_{D}(B_{D, S,u})$, which in turn implies $\braces{v}\cup N_{D-rv}^{-}(v) \subseteq B_{D,S,u}$.
    Furthermore, $S \in \mathfrak{S}_{D}(u)$ implies that $\mathsf{ent}_{D}(B_{D,S,u})$ is linked from $r$ in $D$.
    Thus $B_{D,S,u}$ is an anti-bubble in $D$ that contains $\braces{v} \cup N_{D-rv}^{-}(v)$ and as $A_{D,v}$ is the smallest such anti-bubble by definition, we have $A_{D,v} \subseteq B_{D,S,u}$ as desired.
    To see $ B_{D',S',u} \subseteq B_{D,S,u} $, first note that $S' \subseteq B_{D,S,u}$\,:
    as $B_{D,S,u}$ is the set of vertices that are separated from $r$ by $S$ in $D$, no path from $\mcalP$ can contain a vertex outside of $B_{D,S,u}$, thus as $S' \subseteq V ( \bigcup \mcalP)$ by construction, $S'$ in fact is a subset of $B_{D,S,u}$.
    Now suppose for a contradiction that there is a vertex $w \in B_{D',S',u} \setminus B_{D,S,u}$.
    As $w \not \in B_{D,S,u}$, there is an $r$--$w$ path $Q$ in $D$ that avoids $S$ by definition of $B_{D,S,u}$.
    Since $w \in B_{D',S',u}$, there is no such path in $D'-S'$, which means that $Q$ either meets a vertex of $S'$ or an edge $e \in E(D) \setminus E(D')$.
    In the former case the fact that $S' \subseteq B_{D,S,u}$ together with $S = \mathsf{ent}_{D}(B_{D,S,u})$ which is true by definition, it must meet a vertex of $S$, a contradiction.
    In the latter case $v$ is an internal vertex of $Q$ and as $v \in \mathsf{int}_{D}(B_{D,S,u})$, it cannot be disjoint from $S$, again a contradiction.
    
    Since $\braces{v} \cup N_{D-rv}^{-}(v) \subseteq A_{D,v}$ by definition, every edge in $ E(D)\setminus E(D') $ is spanned by $ A_{D,v} $.
    Thus we can build a $u$-infan in $D'$ that starts in $S$ as follows:
    take the initial segments of the paths in $ \mathcal{P} $ until the first vertex in $ B $ and extend them forward from $B$ to $u$ by using the fact that $ B\in \mathsf{bubb}_{D'}(u) $.
    The resulting path system links $S$ to $u$ in $D'$, as desired.
    \end{proof}

\begin{lem}\label{l:  no collapse of smallest EMsep}
    Assume that $ L $ is large and $ \mathcal{Q}_v\in \mathfrak{P}_{L}(v,S_{L,v} ) $ for some $ v\in V-r $.
    Then $ L':=L\upharpoonright_v E^{+}(\mathcal{Q}_v )$ is large, moreover, $S_{L',u}=S_{L,u} $ for every $ u\in V-r $. 
\end{lem}
\begin{proof}
    Since $ L $ is large, \cref{l: largest EMsep is real} and \cref{cor: bubb in-neigbh} ensure that $ 
    N^{-}_{D-ru}(u) \subseteq B_{L,u}$ for every $ u\in V-r $.
    In particular, all the edges in $ E(L)\setminus E(L') $ are spanned by $ B_{L,v} $.
    We are going to prove that $ B_{L,u}\in \mathsf{bubb}_{L'}(u) $ holds for every $ u\in V-r $.
    Let us first show that this is sufficient. First of all it implies $ B_{L',u}\supseteq B_{L,u} $ for $ u\in V-r $ because $ B_{L',u} $ is the $ \subseteq $-largest element of $ \mathsf{bubb}_{L'}(u) $. By \cref{l: largest EMsep is real} this implies the largeness of $ L' $.
    Let $ u\in V-r $ be given.
    If $ \mathsf{ent}_{L-ru}(B_{L',u})= \mathsf{ent}_{L'-ru}(B_{L',u})$, then $ \mathsf{ent}_{L}(B_{L',u})-u= \mathsf{ent}_{L'}(B_{L',u})-u $ which implies $ B_{L',u}\in \mathsf{bubb}_{L}(u) $ and hence $B_{L,u}\supseteq B_{L',u} $, therefore $ B_{L,u}= B_{L',u} $.
    But then 
    \[
        S_{L,u}=\mathsf{ent}_{L-ru}(B_{L,u})= \mathsf{ent}_{L-ru}(B_{L',u})= \mathsf{ent}_{L'-ru}(B_{L',u})=S_{L',u} \, .
    \]  
    Suppose for a contradiction that $ \mathsf{ent}_{L-ru}(B_{L',u})\neq \mathsf{ent}_{L'-ru}(B_{L',u})$.
    Then we must have $ \mathsf{ent}_{L-ru}(B_{L',u})\supsetneq \mathsf{ent}_{L'-ru}(B_{L',u}) $ with $ \mathsf{ent}_{L-ru}(B_{L',u})\setminus \mathsf{ent}_{L'-ru}(B_{L',u})=\{ v \} $.
    This means that $ v\in \mathsf{int}_{L'-ru}(B_{L',u}) $ and there is some $ wv\in E(L)\setminus E(L') $ with $ w\notin B_{L',u} $.
    Then either $ v\in \mathsf{int}_{L'}(B_{L',u}) $ or $ u=v $.
    Thus by applying \cref{l: bubble unite} with $ B_{L',u} $ and $ B_{L,v}\in \mathrm{bubb}_{L'}(v) $ we conclude that $B_{L',u}\cup B_{L,v}\in \mathrm{bubb}_{L'}(u)$.
    It follows that actually $B_{L',u}\supseteq B_{L,v} $ because $ B_{L',u} $ is the $ \subseteq $-largest element of $ \mathrm{bubb}_{L'}(u) $ by definition.
    But then
    \[
        w\in B_{L,v}\subseteq B_{L',u}\not \ni w \, ,
    \]
    a contradiction.

    Now we turn to the proof of $ B_{L,u}\in \mathsf{bubb}_{L'}(u) $ for every $ u\in V-r $.
    For $ u=v $ this is witnessed by the terminal segments of the paths in $ \mathcal{Q}_v $ from $ S_{L,v} $.
    Let $ u\in V-r-v $ be arbitrary and let $ \mathcal{P} $ be a path-system that witnesses $B_{L,u}\in \mathsf{bubb}_{L}(u) $.
    We can assume that $ \mathcal{P} $ uses some $e\in E(L)\setminus E(L')$, since otherwise $\mathcal{P} $ ensures that $ B_{L,u}\in \mathsf{bubb}_{L'}(u)$ and we are done.
    Let $ P_e $ be the unique path in $ \mathcal{P} $ through $ e $.
    The head $ v $ of $ e $ must be in $ \mathsf{int}_{L}(B_{L,u}) $ because $ V^{-} \mathcal{P}) = \mathsf{ent}_{L}(B_{L,u})-u $ with $ u\neq v $ and the paths in $ \mathcal{P} $ are pairwise disjoint.
    Then $ B_{L,u}\supseteq B_{L,v} $ since otherwise \cref{l: bubble unite} would give $B_{L,u}\subsetneq (B_{L,u}\cup B_{L,v})\in \mathsf{bubb}_{L}(u) $ which is a contradiction.
    We apply \cref{l: aug walk} in $ L' $ with $\mathsf{ent}_{L}(B_{L,u})-u, u $ and $ \mathcal{P}-P_e $.
    If the augmentation is successful, the resulting path-system must lie in $ L'[B_{L,u}] $ and witnesses $B_{L,u} \in \mathsf{bubb}_{L'}(u)$ thus we are done.

    Suppose that the augmentation is unsuccessful, we depict this situation in \cref{fig:lemma_separating_sets_stay_the_same}.
    Then we can choose a unique $ v_P\in V(P)-u $ from each $ P\in \mathcal{P}-P_e $ such that the resulting $ S $ separates $ u $ from $ \mathsf{ent}_{L}(B_{L,u})-u $ in $L'$.
    We know that the terminal segment $vP_eu$ of $ P_e $ must lie in $ B_{L',S,u} $ since otherwise $ S $ would not separate $ u $ from $ \mathsf{ent}_{L}(B_{L,u})-u $ in $ L' $.
    Thus in particular $ v\in B_{L',S,u} $ with $ v\notin S $, i.e.~$ v\in \mathsf{int}_{L'}(B_{L',S,u})$.
    But then $ B:=B_{L',S,u}\cup B_{L,v} \in \mathsf{bubb}_{L'}(u) $ is a subset of $ B_{L,u} $ with $ \mathsf{ent}_{L}(B)-u=\mathsf{ent}_{L'}(B)-u $ because the edges in $ E(L)\setminus E(L') $ are spanned by $ B_{L,v} $.
    Finally, the initial segments of the paths in $ \mathcal{P} $ until the first common vertex with $ B $ lie in $ L' $ and can be forward extended using $ B \in \mathsf{bubb}_{L'}(u) $ to a path-system witnessing $ B_{L,u}\in \mathsf{bubb}_{L'}(u) $.
\end{proof}

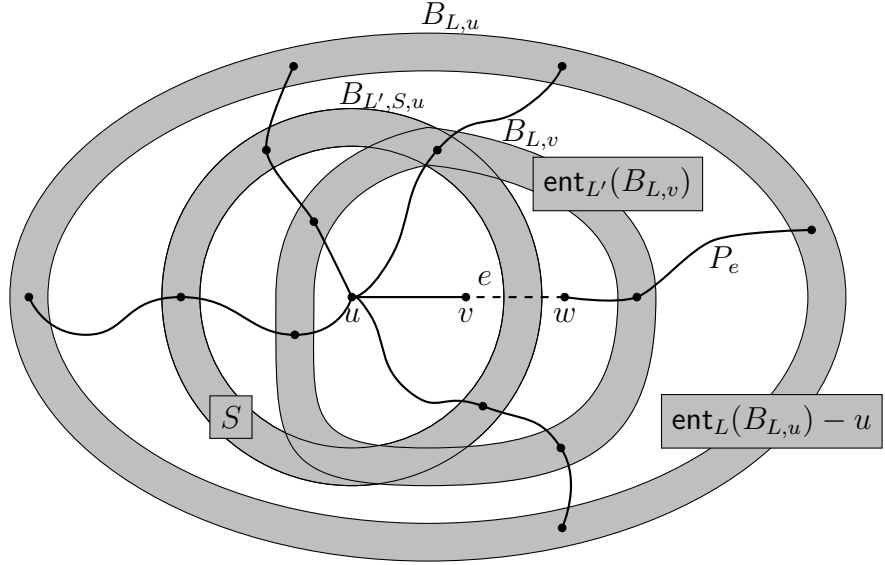
\begin{figure}[ht]
	\centering
	\begin{tikzpicture}[scale=.5]
		\node[shape = coordinate] (u) at (0,0) {};
		\node[shape = coordinate] (v) at (3,0) {};
		\node[shape = coordinate] (el) at (2,0) {};
		\node[shape = coordinate] (x1) at (5.5,-4) {};
		\node[shape = coordinate] (x2) at (7.5,0) {};
		\node[shape = coordinate] (x3) at (-1,2) {};
		\node[shape = coordinate] (x4) at (-1.5,-1) {};

		\node[ellipse,draw,fill = lightgray,minimum width = 11cm, minimum height = 7cm] (BL) at (el) {};
		\node[ellipse,draw, outer sep = .25cm, fill = white,minimum width = 10cm, minimum height = 6cm] (BL) at (el) {};
		\node[circle,draw,minimum size = 5cm,fill = lightgray] (outeru) at (u){};
		\node[circle,draw,minimum size = 4cm, outer sep = .25cm, fill = white] (inneru) at (u){};
		\draw[fill,lightgray] plot [smooth,tension=1.5] coordinates { (2,4.5) (8,0) (2,-5) (-2,0) (2,4.5) } -- plot [smooth,tension=1.5] coordinates { (2,3.5) (-1,0) (2,-4) (7,0) (2,3.5) };
		\draw plot [smooth,tension=1.5] coordinates { (2,4.5) (8,0) (2,-5) (-2,0) (2,4.5) };
		\draw plot [smooth,tension=1.5] coordinates { (2,3.5) (-1,0) (2,-4) (7,0) (2,3.5) };
		\node[circle,draw,minimum size = 5cm] at (u){};
		\node[circle,draw,minimum size = 4cm, outer sep = .25cm] at (u){};
		
		\draw[thick,green!90!black] plot [smooth, tension = 1] coordinates { (BL.300) (x1) (inneru.320) (1.5,-2.5) (.5,-.5) (u) };
		\draw[thick,blue] plot [smooth, tension = .5] coordinates { (BL.10)  (9.5,1.5) (x2) (inneru.10) (v) (u) };
		\draw[thick,green!90!black] plot [smooth, tension = 1] coordinates { (BL.60) (4.5,5) (inneru.60) (1,1) (u) };
		\draw[thick,green!90!black] plot [smooth, tension = .5] coordinates { (BL.120) (inneru.120) (x3) (u) };
		\draw[thick,green!90!black] plot [smooth, tension = 1] coordinates { (BL.180) (-7,-1) (inneru.180) (x4) (u) };
		
		\node[below] at (u) {$u$};
		\node[below] at (v) {$v$};
		\node[blue] at (9.8,1) {$P_{e}$};
		\node at (3.3 ,0.5) {$e$};
		\draw (BL.340) node [draw,fill=lightgray] {$\mathsf{ent}_{L}(B_{L,u})-u$};
		\draw (inneru.225) node [draw,fill=lightgray] {$S$};
		\draw (7,3) node [draw,fill=lightgray] {$\mathsf{ent}_{L^\prime}(B_{L,v})$};
		
		\foreach \x in {u,v,x1,x2,x3,x4} \draw[fill] (\x) circle [radius=.1];
		\foreach \x in {10,60,120,180,320} \draw[fill] (inneru.\x) circle [radius=.1];
		\foreach \x in {10,60,120,180,300} \draw[fill] (BL.\x) circle [radius=.1];
	\end{tikzpicture}
	\caption{The situation when the augmentation in the proof of \cref{l:  no collapse of smallest EMsep} is unsuccessful.}
    \label{fig:lemma_separating_sets_stay_the_same}
    \end{figure}

\section{Proof of the main result}
\subsection{Definitions and a sketch of the construction}
Let an $ r $-rooted digraph $ D=(V,E) $ of size $ \aleph_1 $ be fixed.
First of all, we can assume by \cref{cor: wlog D is quasi-flame} that $ D $ is a quasi-vertex-flame.
Let $ \left\langle M_\alpha: \alpha \leq\omega_1 \right\rangle $ be a sequence such that
\begin{itemize}
    \item $ M_0=\varnothing $;
    \item $ M_\alpha $ is an elementary submodel for each $\alpha>0 $;
    \item $ M_\alpha=\bigcup_{\beta<\alpha}M_\beta $ if $ \alpha $ is a limit ordinal;
    \item $ M_\alpha $ is countable for $ \alpha<\omega_1 $;
    \item $D, r, \left\langle M_\beta: \beta\leq\alpha \right\rangle, \alpha\in M_{\alpha+1} $ for $ \alpha<\omega_1 $.
\end{itemize}
  
\begin{obs}\label{obs: elemtary contains}
    $ M_\beta\cup \{ M_\beta \}\subseteq M_\alpha $ for $ \beta<\alpha\leq \omega_1 $.
\end{obs}
Let $ \boldsymbol{M^{\alpha}}:=M_{\alpha+1}\setminus M_\alpha $ for $ \alpha<\omega_1 $ and we define $ \boldsymbol{V_\alpha}:=V\cap M_\alpha,\boldsymbol{D_\alpha}:=D\cap M_\alpha=D[V_\alpha] $ for $\alpha \leq \omega_1 $ as well as
\begin{align*}
    \boldsymbol{V^{\alpha}}&:=V\cap M^{\alpha},\\
    \boldsymbol{D^{\alpha}}&:=(D\cap M_{\alpha+1})\setminus \left[E(D_\alpha)\cup \mathsf{out}_D(V_\alpha-r) \right] , \text{ and}\\
    \boldsymbol{D^{\alpha\leq}}&:=D\setminus \left[E(D_\alpha)\cup \mathsf{out}_D(V_\alpha-r) \right] 
\end{align*} 
for $ \alpha<\omega_1 $.
We choose enumerations $ V_{\alpha}-r=\{\boldsymbol{v_{\alpha,n}} \colon n<\omega\} $ and $ V^{\alpha}-r=\{\boldsymbol{v^{\alpha}_{n}} \colon n<\omega \} $ for $ \alpha<\omega_1 $ (technically we fix a choice function $ c\in M_1 $ and we choose the enumerations accordingly).
Recall that every countable ordinal number can be written uniquely in the form $ \omega \alpha+n $ where $ n<\omega $.
We obtain an enumeration $ V-r=\{ v_\xi: \xi<\omega_1 \} $ by letting $ \boldsymbol{v_{\omega \alpha +n}}:=v^{\alpha}_n $.
Observe that $ V_{\alpha}-r=\{ v_\xi: \xi<\omega \alpha \} $ for $ \alpha\leq \omega_1 $.
We shall construct a sequence $ \left\langle L_\xi: \xi\leq\omega_1 \right\rangle $ of large subdigraphs of $ D $ with
\begin{itemize}
    \item $ L_0=D $;
    \item we obtain $ L_{\xi+1} $ by the deletion some of the incoming edges of $ v_\xi $ from $ L_\xi $;
    \item $ L_\nu=\bigcap_{\xi<\nu}L_\xi $ if $ \nu $ is a limit ordinal.
\end{itemize}
Before giving the complete list of properties of the recursive construction, we need some definitions.
We also explain roughly the ideas behind what we are going to do in order to make it easier to follow the formal proof afterwards.
First of all, $ L_{\omega_1} $ will be a large vertex-flame which completes the proof of \cref{t: main result}.
For $ v=v_{\omega \beta +n} $, let $ \boldsymbol{S_v}:=S_{L_{\omega \beta},v} $.
Note that after $ L_{\omega \beta} $ is defined, the separations $ S_v $ will be defined for the following countably many vertices, namely for the vertices in $ V^{\beta} $.
By guaranteeing that $ L_\xi $ is large for every $ \xi $, we will automatically ensure $ S_v\in \mathfrak{S}_{D,v} $ by \cref{l: largest EMsep is real}.
We strive to end up with path-systems $ \mathcal{P}_v\in \mathfrak{P}_{L_{\omega_1}}(v, S_v) $ with $E^{+}(\mathcal{P}_v) = \mathsf{in}_{L_{\omega_1}}(v)-rv$ in $ L_{\omega_1} $ for every $ v\in V-r $.
Note that the path-system $ \mathcal{P}_v $ promised in the abstract is can be chosen to be the $ \mathcal{P}_v $ we are about to construct whenever $ rv\notin E(D) $, otherwise the single-edge path $ rv $ need to be added to get a promised path-system.
The Path-systems $ \mathcal{P}_v $ witness that $ L_{\omega_1} $ is indeed a large vertex-flame.
For each $ v\in V-r $ we build the path-system $ \mathcal{P}_v $ ``layer by layer'' according to our chain of elementary submodels (see \cref{figure: layer}) in the following sense.
If $ v=v_{\omega \beta+n} $, then we construct first a segment $\mathcal{P}_{v,\beta+1}\in \mathfrak{P}_{L_{\omega_1}\cap D_{\beta+1}}(v,S_v\cap V_{\beta+1}) $ with $E^{+} \mathcal{P}_{v,\beta+1})= \mathsf{in}_{L_{\omega_1}\cap D_{\beta+1}}(v)-rv $. 
In every new layer we extend this by a new segment: For every $ \gamma $ with $\beta+1\leq \gamma <\omega_1 $ we construct a path-system $\mathcal{P}^{\gamma}_v\in \mathfrak{P}_{L_{\omega_1}\cap D^{\gamma}}(v, S_v\cap V^{\gamma}) $ with $E^{+}(\mathcal{P}^{\gamma}_{v}) = \mathsf{in}_{L_{\omega_1}\cap D^{\gamma}}(v) $.
Since by the definition of $ D^{\gamma} $ the paths in $\mathcal{P}^{\gamma}_v $ are internally disjoint from $ V_\gamma $, the new segments never share any internal vertex with the already constructed segments.
By letting
\[ 
    \mathcal{P}_{v,\alpha}:=\mathcal{P}_{v,\beta+1}\cup\bigcup_{\beta+1\leq\gamma<\alpha} \mathcal{P}^{\gamma}_{v}
\]
for $\beta+1\leq \alpha \leq \omega_1 $, the path-system $ \mathcal{P}_v:=\mathcal{P}_{v,\omega_1} $ will be as desired.

Instead of constructing $\mathcal{P}_{v,\beta+1} $ and $\mathcal{P}^{\gamma}_v $ ``directly'' we are going to build some supersets of them and throw away the surplus paths.
For this we let $ \boldsymbol{L^{\omega \beta+ n}}:= L_{\omega \beta+ n}\cap D^{\beta\leq}$.

\begin{figure}[H]
    \centering
    \begin{tikzpicture}
            \node[shape = coordinate] (r) at (.3,1.9) {};
            \node[shape = coordinate] (v) at (1.9,.3) {};
            \node[shape = coordinate] (outr1) at (.7,2.4) {};
            \node[shape = coordinate] (outr2) at (.9,2.4) {};
            \node[shape = coordinate] (outr3) at (.5,2.4) {};
            \node[shape = coordinate] (outr4) at (-.4,1.8) {};
            \node[shape = coordinate] (outr5) at (-.4,1.6) {};
            \node[shape = coordinate] (outv1) at (2.3,.6) {};
            \node[shape = coordinate] (outv2) at (2.3,.8) {};
            \node[shape = coordinate] (outv3) at (2.3,.4) {};
            \node[shape = coordinate] (outv4) at (1.8,-.4) {};
            \node[shape = coordinate] (outv5) at (1.6,-.4) {};
            \node[shape = coordinate] (s1) at (1,1) {};
            \node[shape = coordinate] (s2)  at (1.2,1.2) {};
            \node[shape = coordinate] (s3)  at (1.4,1.4) {};
            \node[shape = coordinate] (s4)  at (2.5,2.4) {};
            \node[shape = coordinate] (s5)  at (2.8,2.6) {};
            \node[shape = coordinate] (s6)  at (4.3,4) {};
            \node[shape = coordinate] (s7)  at (4.8,4.4) {};
            \node[shape = coordinate] (s8) at (3.5,3.2) {};
            \node[shape = coordinate] (s9) at (5.2,5) {};
            \node[shape = coordinate] (s10) at (5.6,5.3) {};

            \draw[green!60!black,dashed]  plot[smooth cycle, tension=.7] coordinates {(0.8,1) (1.4,1.6) (2.2,2.4) (2.8,3) (3.4,3.6) (4.2,4.4) (5.6,5.8) (6,5.4) (4.8,3.8) (3.8,3) (2,1.6) (1.2,0.8)};

            \draw plot [smooth cycle] coordinates { (0,0) (1,-.2)  (2,0)  (2,2) (0,2) (-.2,1) };
            \draw plot [smooth] coordinates { (2,0) (4,0) (4,4) (0,4) (0,2) };
            \draw plot [smooth] coordinates { (1,-.2) (1,-1) (5,-1) (5,5) (-1,5) (-1,1) (-.2,1) };
            \draw[dotted] plot [smooth] coordinates { (5,-1) (6,-.6) (6,6) (-.6,6) (-1,5) };

            \draw[blue] plot [smooth] coordinates { (r) (1.4,1.4) (v) };
            \draw[blue] plot [smooth] coordinates { (r) (1.2,1.2) (v) };
            \draw[blue] plot [smooth] coordinates { (r) (1,1) (v) };
            
            \draw[red] plot [smooth, tension = 1.2] coordinates { (outr1) (1.3,2.8) (2.5,2.9) (3.3,1) (outv1) };
            \draw[red] plot [smooth, tension = 1.2] coordinates { (outr2) (1.5,2.7) (2.5,2.4) (2.9,1.2) (outv2) };
            \draw[red] plot [smooth, tension = 1.2] coordinates { (outr3) (1.5,3.7) (s8) (3.53,0.5) (outv3) };
            
            \draw[green] plot [smooth] coordinates { (outr4) (-.2,4.6) (s6) (4,-.2) (outv4) };
            \draw[green] plot [smooth] coordinates { (outr5) (-.8,2) (-.3,4.8) (s7) (4.3,-.4) (2,-.8) (outv5) };
            
            \foreach \y in {r,v} \foreach \x in {1,2,3,4,5} \draw (\y) to (out\y\x);
            
            \node[below] at (r) {$r$};
            \node[left] at (v) {$v$};
            \node at (-.3,-.3) {$D_{\beta+1}$};
            \node at (.4,3.6) {$D^{\beta+1}$};
            \node at (-2,3.5) {$D^{\beta+2}$};
            \node at (6.7,-.6) {$D^{\beta+2\leq}$};
            \node[blue] at (.7,.7) {$\mathcal{P}_{v,\beta+1}$};
            \node[red] at (3.7,.5) {$\mathcal{P}_{v}^{\beta+1}$};
            \node[green] at (4.3,-1) {$\mathcal{P}^{\beta+2}_{v}$};
            \node[green!60!black] at (3.2,3) {$S_v$};

            \foreach \x in {r,v,outr1,outr2,outr3,outr4,outr5,outv1,outv2,outv3,outv4,outv5, s1, s2, s3, s4, s5, s6, s7, s8,s9,s10} \draw[fill] (\x) circle [radius=.05];
    \end{tikzpicture}
    \caption{A sketch of the strategy to build $ \mathcal{P}_v $.}\label{figure: layer}
    \label{fig:general_approach}
\end{figure}
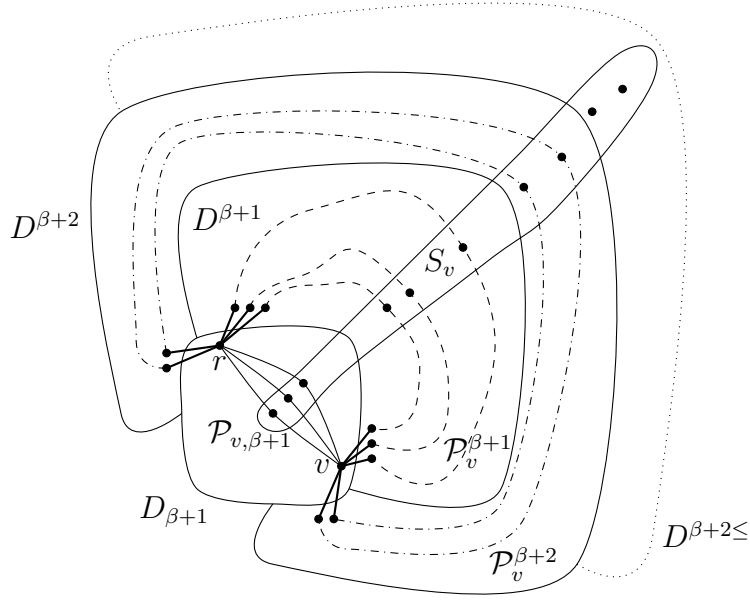

\subsection{The conditions of the recursion}
Let us now make the construction precise.
We shall define by transfinite recursion sequences 
\[
    \left\langle \mathcal{Q}_{\xi}: \xi<\omega_1 \right\rangle,\ \left\langle \mathcal{Q}^{\xi}: \omega\leq \xi<\omega_1 \right\rangle\text{ and } \left\langle L_\xi: \xi \leq \omega_1 \right\rangle
\]
satisfying the following properties:
\begin{enumerate}[label={(\arabic*)},ref={(\arabic*)},font=\upshape]
    \item\label{item: main 1} $ L_0=D $;
    \item\label{item: main 2} $\mathcal{Q}_{\xi}\in \mathfrak{P}_{L_\xi}(v_\xi, S_{v_\xi}) $;
    \item\label{item: main 3} $ L_{\xi+1}:=L_{\xi} \upharpoonright_{v_{\xi} } E^{+}(\mathcal{Q}_{\xi}) $;
    \item \label{item: main 4} $ L_\nu=\bigcap_{\xi<\nu}L_\xi $ if $ \nu $ is a limit ordinal;
    \item \label{item: main 5} $ L_\xi $ is $D$-large;
    \item \label{item: main 6} $ S_{L_{\omega \alpha+n},v}= S_v$ for every $\omega \alpha+n<\omega_1 $ and $ v\in 
    V^{\alpha} $;
    \item\label{item: main 7} For every $\omega \alpha+n<\omega_1 $ and $ v\in V_\alpha-r $: 
    \begin{enumerate}[label={(\alph*)},ref={(7\hspace{1pt}\alph*)},font=\upshape]
        \item\label{item: main 7a} $ S_{v}\setminus V_\alpha\in 
        \mathfrak{S}_{L^{\omega \alpha+n} }(v) $;
        \item\label{item: main 7b} $ L^{\omega \alpha+n} $ has the vertex-flame property at $ v$;
        \item\label{item: main 7c} If $ v=v_{\alpha, n} $, then $ \mathcal{Q}^{\omega\alpha+n}\in \mathfrak{P}_{L^{\omega 
        \alpha+n}}(v,S_{v}\setminus 
        V_\alpha) $ with 
        \[ E^{+}(\mathcal{Q}^{\omega\alpha+n})=\mathsf{in}_{L^{\omega \alpha+n}}(v); \]
    \end{enumerate}
    \item\label{item: main 8}
    $\left\langle\mathcal{Q}_{\xi}:\xi <\nu \right\rangle,\ \left\langle \mathcal{Q}^{\xi}: \omega\leq \xi <\nu \right\rangle,\ 
    \left\langle L_\xi: \xi < \nu \right\rangle\in M_{\alpha+1}$ for $\nu=\omega\alpha +n<\omega_1$; 
    \item\label{item: main 9} For $ v=v_{\omega\beta+m} $:
    \begin{align*}
        \bigcup_{n<m} \mathsf{in}_{Q_{\omega\beta+n}}(v) &\subseteq E^{+}(\mathcal{Q}_{\omega\beta+m})+rv\text{ if } 
        \beta=0 \text{ and }\\
        \bigcup_{n\leq m}\mathsf{in}_{Q^{\omega\beta+n}}(v)\cup \bigcup_{n<m} \mathsf{in}_{Q_{\omega\beta+n}}(v) 
        &\subseteq E^{+}(\mathcal{Q}_{\omega\beta+m})+rv\text{ if } \beta>0. 
    \end{align*}
\end{enumerate}
Note that $ L_\nu $ is uniquely determined by $ \left\langle \mathcal{Q}_\xi: \xi<\nu \right\rangle $ (see properties \labelcref{item: main 1,item: main 2,item: main 3,item: main 4}, thus we are going to always have  at most one suitable choice for $ L_\xi $ but we still need to check if it respects the conditions.
The preservation of property \labelcref{item: main 8} will follow immediately from the fact that the definitions of $ \left\langle\mathcal{Q}_{\xi} \colon \xi <\omega \alpha+n \right\rangle $ and $ \left\langle \mathcal{Q}^{\xi} \colon \omega\leq \xi < \omega \alpha+n \right\rangle $ rely only on parameters that are in $ M_{\alpha+1} $, namely $D, r,\ \left\langle M_\beta, \beta \leq \alpha \right\rangle,\ c $ and vertices $ v_{\omega \alpha+k}, v_{\alpha,k} $ for $ k<n $ (where $ c $ is some fixed choice function).

\subsection{The path-systems \texorpdfstring{$ \boldsymbol{\mathcal{P}_{v, \beta+1}} $}{Pvb} and 
\texorpdfstring{$ \boldsymbol{\mathcal{P}^{\gamma}_{v}} $}{Pcv}}
Let $ \nu\leq \omega_1 $ and suppose that $ \mathcal{Q}_\xi $ and $ L_\xi $ are defined for $ \xi<\nu $ and $ \mathcal{Q}^{\xi} $ is defined for $ \omega \leq \xi<\nu $ and none of the conditions \labelcref{item: main 1,item: main 2,item: main 3,item: main 4,item: main 5,item: main 6,item: main 7,item: main 8,item: main 9} is violated so far.
Let $ v=v_{\omega \beta+n} $ for some $ \omega \beta+n<\nu $ be fixed. 
We define $\boldsymbol{\mathcal{P}_{v, \beta+1}}:= \mathcal{Q}_{\omega \beta+n}\cap M_{\beta+1} $.
By the definition of enumerations $ \{ v_{\gamma,n}: n<\omega \} $, for every $ \gamma $ with $\beta+1\leq \gamma<\omega_1$ there is a unique $ m_{\gamma}<\omega $ with $ v_{\gamma, m_\gamma}=v $.
We let $ \boldsymbol{\mathcal{P}^{\gamma}_{v}}:=\mathcal{Q}^{\omega\gamma+m_\gamma}\cap M_{\gamma+1} $ whenever $\beta+1\leq \gamma$ and $ \omega\gamma+m_\gamma<\nu $.

Property \labelcref{item: main 9} is designed to prevent the deletion of edges of the path-systems $\mathcal{P}_{v,\beta+1} $ and $ \mathcal{P}^{\gamma}_{v} $:
\begin{lem}\label{l: preserving path-system}
    Let $ v=v_{\omega \beta+n} $ for some $ \omega \beta+n<\nu $.
    The path-systems $\mathcal{P}_{v,\beta+1} $ and $ \mathcal{P}^{\gamma}_{v} $ lie in $ L_{\xi}$ for every $ \xi<\nu $.
\end{lem}
\begin{proof}
    Since $ \mathcal{Q}_{\omega \beta+n} $ is a path-system in $ L_{\omega \beta+n} $ (see property \labelcref{item: main 2}) so is $\mathcal{P}_{v,\beta+1} $.
    It follows from properties \labelcref{item: main 1,item: main 2,item: main 3,item: main 4} that $ L_\xi $ is a $ \subseteq $-decreasing function of $ \xi $, this implies that $\mathcal{P}_{v,\beta+1} $ is a path-system in $ L_\xi $ for $ \xi\leq \omega \beta+n $.
    By property \labelcref{item: main 3}, we do not delete any edges of $ \mathcal{Q}_{\omega \beta+n}\supseteq \mathcal{P}_{v,\beta+1}$ when we obtain $ L_{\omega \beta+n+1} $ from $ L_{\omega \beta+n} $.
    Property \labelcref{item: main 9} applied to vertices $ v_{\omega \beta+m} $ with $ n<m<\omega $ together with property \labelcref{item: main 3} guarantees that none of the edges of $\mathcal{P}_{v,\beta+1}$ is deleted when we construct $ L_{\omega \beta+m} $ for $ m>n $.
    Thus $\mathcal{P}_{v,\beta+1} $ is a path-system in $ L_{\omega(\beta+1)} $ as well by property \labelcref{item: main 4}.
    Whenever $ P\in M_{\beta+1} $ is a path, we have $ V(P)\subseteq M_{\beta+1} $, therefore $\mathcal{P}_{v,\beta+1} $ lies completely in $ D_{\beta+1} $ and after step $ \omega(\beta+1) $ we delete only edges $ e $ whose head is in $ V\setminus V_{\beta+1} $. 
    Thus the path-system $\mathcal{P}_{v,\beta+1} $ lies in $ L_\xi $ for every $ \xi $ with $ \omega(\beta+1)<\xi<\nu $ as well.
    The proof for $ \mathcal{P}^{\gamma}_{v} $ goes similarly.
\end{proof}
 
\begin{prop}\label{prop: P v beta plus 1}
    If $ v=v_{\omega \beta+n} $ for some $ \omega \beta+n<\nu $, then $\mathcal{P}_{v,\beta+1}\in \mathfrak{P}_{D_{\beta+1}}(v, S_v\cap V_{\beta+1})$.
    Furthermore, $ \mathsf{in}_{L_{\xi}}(v)\cap E(D_{\beta+1})-rv=E^{+}(\mathcal{P}_{v,\beta+1}) $ for every $\xi\in (\omega \beta+n, \nu) $.
\end{prop}
\begin{proof}
    Note that $ \omega \beta+n\in M_{\beta+1} $ because $ \beta\in M_{\beta+1} $ by assumption.
    Hence by property~\labelcref{item: main 8}, $\mathcal{Q}_{\omega \beta+n}\in M_{\beta+1} $.
    Each $P\in \mathcal{Q}_{\omega \beta+n} $ who has an internal vertex $ u $ in $ V_{\beta+1}$ is definable from $ \mathcal{Q}_{\omega \beta+n} $ and $ u $ and therefore must be in $ M_{\beta+1} $.
    This means that for each $P\in \mathcal{Q}_{\omega \beta+n} $ either $ V(P)\subseteq V_{\beta+1} $ or $ P $ is internally disjoint from $ V_{\beta+1} $.
    Thus by property~\labelcref{item: main 2}, it follows that $\mathcal{P}_{v,\beta+1}\in\mathfrak{P}_{D_{\beta+1}}(v, S_v\cap V_{\beta+1}) $.
    Moreover, by property \labelcref{item: main 3}, $ \mathcal{P}_{v,\beta+1} $ covers $\mathsf{in}_{L_{\omega \beta+n+1}}(v)\cap E(D_{\beta+1})-rv $ which is the same as $\mathsf{in}_{L_{\xi}}(v)\cap E(D_{\beta+1})-rv $ whenever $\omega \beta+n+1\leq \xi<\nu $ by properties \labelcref{item: main 1,item: main 2,item: main 3,item: main 4}.
\end{proof}

\begin{prop}\label{prop: P v gamma}
    If $ v=v_{\gamma, m_\gamma}=v_{\omega\beta+n} $ with $ \omega \gamma+m_\gamma<\nu $, then $\mathcal{P}_{v}^{\gamma}\in \mathfrak{P}_{D^{\gamma}}(v,S_{v}\cap V^{\gamma}) $.
    Furthermore, $ \mathsf{in}_{L_\xi\cap D^{\gamma}}(v)=E^{+}(\mathcal{P}_{v}^{\gamma}) $ for every $ \xi\in (\omega\beta+n, \nu) $.
\end{prop}
\begin{proof}
    Note that $ \beta< \gamma $ because $ v\in V_\gamma $ by $ v=v_{\gamma, m_\gamma} $ and $ \beta+1 $ is the smallest ordinal with $ v\in V_{\beta+1} $ by $ v=v_{\omega\beta+n} $ (see the definition of the enumerations after \cref{obs: elemtary contains}).
    Since $ \mathcal{Q}^{\omega\gamma+m_\gamma}\in M_{\gamma+1} $, we obtain via property \labelcref{item: main 7c} that $ \mathcal{P}_{v}^{\gamma} \in \mathfrak{P}_{D^{\gamma}}(v,(S_{v}\setminus V_{\gamma})\cap V_{\gamma+1}) $ and $ E^{+} \mathcal{P}_{v}^{\gamma})=\mathsf{in}_{L^{\omega \gamma+m_\gamma}\cap D_{\gamma+1}}(v) $.
    The first part of the proposition follows by observing that $ (S_{v}\setminus V_{\gamma})\cap V_{\gamma+1}=S_v\cap V^{\gamma} $ by definition.
    The second part follows from the fact that: $ L^{\omega \gamma+m_\gamma}\cap D_{\gamma+1}=L^{\omega \gamma+m_\gamma}\cap D^{\gamma} $ (which is also a direct consequence of the corresponding definitions).
    Finally $ \mathsf{in}_{L^{\xi}}(v) $ remains the same for every $ \xi\in (\omega \beta+n, \nu) $ by properties \labelcref{item: main 1,item: main 2,item: main 3,item: main 4}.
\end{proof}
 
\begin{lem}\label{l: rest of Q}
    If $ v=v_{\omega \beta+n} $ for some $ \omega \beta+n<\nu $ and $\beta<\alpha < \omega_1 $, then we have $\mathcal{Q}_{\omega \beta+n}\setminus M_\alpha\in \mathfrak{P}_{L^{\omega \alpha}}(v, S_v\setminus V_\alpha)$.
    Furthermore, $ \mathsf{in}_{L_{\xi}}(v)\cap E(D^{\alpha \leq}) =E^{+}(\mathcal{Q}_{\omega   \beta+n}\setminus M_\alpha)$ for every $\xi\in (\omega \beta+n, \nu) $.
\end{lem}
\begin{proof}
    First of all, $ \mathcal{Q}_{\omega \beta+n} $ lies in $ L_{\omega \beta+n} $ according to property \labelcref{item: main 2} and we have already seen that $ \mathcal{Q}_{\omega \beta+n}\in M_{\beta+1}\subseteq M_\alpha $.
    Thus each path in $ \mathcal{Q}_{\omega \beta+n}\setminus M_\alpha $ is internally disjoint from $ V_{\alpha} $. 
    This gives via property \labelcref{item: main 2} that $\mathcal{Q}_{\omega \beta+n}\setminus M_\alpha\in \mathfrak{P}_{L^{\omega \beta+n}}(v, S_v\setminus V_\alpha)$.
    In order to obtain $ L^{\omega \alpha} $ from $ L^{\omega \beta+n} $, we delete only edges whose heads are in $V_\alpha-r $.
    The only such edges in $ E(\mathcal{Q}_{\omega \beta+n}\setminus M_\alpha) $ are also in $ E^{+}( \mathcal{Q}_{\omega \beta+n}) $ but we do not delete any of those by property \labelcref{item: main 3}, therefore $ \mathcal{Q}_{\omega \beta+n}\setminus M_\alpha\in \mathfrak{P}_{L^{\omega \alpha}}(v, S_v\setminus V_\alpha) $.
    The second part follows directly from properties \labelcref{item: main 1,item: main 2,item: main 3,item: main 4}.
\end{proof}

\subsection{Limit step}
Suppose now that $ \nu=\omega \alpha $ for some $ \alpha \leq \omega_1 $ and as earlier assume that $ \mathcal{Q}_\xi $ and $ L_\xi $ are defined for $ \xi<\nu $ and $ \mathcal{Q}^{\xi} $ is defined for $ \omega \leq \xi<\nu $ and none of the conditions \labelcref{item: main 1,item: main 2,item: main 3,item: main 4,item: main 5,item: main 6,item: main 7,item: main 8,item: main 9} is violated so far.
If $ \alpha=0 $, then let $ L_0:=D $ which is our only possible choice according property \labelcref{item: main 1} and this choice does not violate any of the conditions for trivial reasons.
If $ \alpha>0 $, then our only option is to let $ L_{\omega \alpha}:=\bigcap_{\xi<\omega \alpha}L_\xi $ (see property \labelcref{item: main 4}).
We need to check that none of \labelcref{item: main 1,item: main 2,item: main 3,item: main 4,item: main 5,item: main 6,item: main 7,item: main 8,item: main 9} is violated.
Note that $\mcalQ^{\omega \alpha}$ and $\mcalQ_{\omega \alpha}$ will be defined from $L_{\omega \alpha}$ in step $\omega \alpha + 1$, thus we need not check any clause that refers to either of them.
The preservation of \labelcref{item: main 1,item: main 2,item: main 3,item: main 4} is clear.
We are intended to apply \cref{cor: presLare} to demonstrate the largeness of $ L_{\omega \alpha} $ (property \labelcref{item: main 5}).
Clearly $ \mathsf{out}_D(r)\subseteq L_{\omega \alpha} $ because we never delete any outgoing edge of $ r $ (this was built into the definition of ``$ \upharpoonright_v $'').
Note that $ \mathsf{in}_{L_{\omega \alpha}}(v)\subsetneq \mathsf{in}_D(v) $ may happen only for $ v\in V_\alpha-r $.
For every $ v\in V_\alpha-r $ there is some $ \beta<\alpha $ and $ n<\omega $ such that $ v=v_{\omega \beta+n} $.
We know by property \labelcref{item: main 5} that $ L_{\omega \beta} $ is large but then $S_{L_{\omega \beta}, v} $ (which is $ S_v $ by definition) is in $ \mathfrak{S}_{D}(v) $ (see \cref{l: largest EMsep is real}).
We define 
\[
    \boldsymbol{\mathcal{P}_{v,\alpha}}:=\mathcal{P}_{v,\beta+1}\cup\bigcup_{\beta+1\leq\gamma<\alpha} \mathcal{P}^{\gamma}_{v} \, .
\]
Note that $ \mathsf{in}_{L_{\omega \alpha}}(v)\cap E(D_{\beta+1})-rv=E^{+}(\mathcal{P}_{v,\beta+1}) $ and $ \mathsf{in}_{L_{\omega \alpha}\cap D^{\gamma}}(v)=E^{+}(\mathcal{P}_{v}^{\gamma}) $ follow from \cref{prop: P v beta plus 1} and \cref{prop: P v gamma} respectively via $ L_{\omega\alpha}=\bigcap_{\xi<\nu}L_\xi $.
Thus these propositions have the following consequence:
\begin{cor}\label{cor: P v alpha}
For every $ v\in V_\alpha-r $, $ \mathcal{P}_{v,\alpha}\in \mathfrak{P}_{D_{\alpha}}(v, S_v\cap V_{\alpha}) $ with $ E^{+}(\mathcal{P}_{v, \alpha})= \mathsf{in}_{L_{\omega \alpha}}(v)\cap E(D_{\alpha})-rv $.
\end{cor}

It follows from \cref{l: rest of Q} and \cref{cor: P v alpha} that if $ v=v_{\omega \beta+n}\in V_\alpha-r $, then
\[
    \mathcal{P}_{v, \alpha}\cup (\mathcal{Q}_{\omega \beta+n}\setminus M_\alpha)\in \mathfrak{P}_{L_{\omega \alpha}}(v, S_v) \, .
\]
Thus $ L_{\omega \alpha} $ is large by \cref{cor: presLare}.
Property \labelcref{item: main 6} for $ n=0 $ is true by the definition of $ S_v $.
Property \labelcref{item: main 7a} and \labelcref{item: main 7b} for $ v=v_{\omega\beta+n} $ is witnessed by $ \mathcal{Q}_{\omega \beta+n}\setminus M_\alpha $ according to \cref{l: rest of Q}.
Property \labelcref{item: main 8} is maintained, because as we already argued the transfinite recursion so far can be carried out in $ M_{\alpha+1} $ since it relies only on the parameters $ D, r, \left\langle \mathcal{Q}_\xi \colon \xi<\omega \alpha \right\rangle, c\in M_{\alpha+1} $.
Finally, we do not check \labelcref{item: main 7c} and \labelcref{item: main 9}, since they refer to $\mcalQ^{\omega \alpha}$.

\begin{lem}\label{l: lim linked U from r}
    $ L_{\omega \alpha}\cap D_\alpha $ is a $ D_\alpha $-large vertex-flame.
    In particular, if $ \alpha=\omega_1 $, then $ L_{\omega_1} $ is a large vertex-flame.
\end{lem}
\begin{proof}
We know by \cref{cor: P v alpha} that $ \mathcal{P}_{v,\alpha}\in \mathfrak{P}_{D_{\alpha}}(v, S_v\cap V_{\alpha}) $ with $ E^{+}(\mathcal{P}_{v, \alpha})= \mathsf{in}_{L_{\omega \alpha}}(v)\cap E(D_{\alpha})-rv$.
Note that $ S_v\cap V_{\alpha} $ separates $ v $ from $ r $ in $ D_{\alpha}-rv $ because so does $ S_v $ in $ D-rv $.
Thus $ S_v\cap V_\alpha\in \mathfrak{S}_{D_\alpha}(v) $ witnessed by $ \mathcal{P}_{v,\alpha} $. Since the path-systems $ \mathcal{P}_{v,\alpha} $ for $ v\in V_\alpha-r $ lie in $ L_{\omega \alpha} $ (see \cref{l: preserving path-system} and \labelcref{item: main 4}) they show that $ L_{\omega \alpha}\cap D_\alpha $ is a $ D_\alpha $-large vertex-flame.
\end{proof}
We will make use of the following consequence.
\begin{cor}\label{cor: linked from r finite}
    For every finite $ U\subseteq V_\alpha-r $ which is linked from $ r $ in $ D $, it is linked from $ r $ also in $ L_{\omega \alpha}\cap D_\alpha $.
\end{cor}
\begin{proof}
This follows directly from \cref{l: lim linked U from r} via \cref{l: linked from r U}.
\end{proof}

\subsection{Successor step}
Suppose that there is some $ \omega\alpha+n<\omega_1 $ such that the following  are already defined without violating the conditions:
\begin{itemize}
    \item $ L_\xi $ for $ \xi \leq \omega\alpha+n $;
    \item $ \mathcal{Q}_\xi $ for $ \xi <\omega\alpha+n $;
    \item $ \mathcal{Q}^{\xi} $ for $\omega\leq \xi < \omega\alpha+n $.
\end{itemize}

Let $ v=v_{\omega\alpha+n} $.
Note that $ L_{\omega\alpha+n} $ is a quasi-vertex-flame by \cref{cor: large for a quasi is quasi} because $ D $ is a quasi-vertex-flame by assumption and $ L_{\omega\alpha+n} $ is large by property \labelcref{item: main 5}.
Suppose first that $ \alpha=0 $.
Let $ I:=\bigcup_{k<n} \mathsf{in}_{\mathcal{Q}_{k}}(v) $.
Since $ \left|I\right|\leq n $, we have $ I\in \mathcal{G}_{L_n}(v) $ by the quasi-vertex-flame property. We have $ S_n\in \mathfrak{S}_{D}(v) $ by property \labelcref{item: main 6}.
By applying \cref{cor: EMpathsys covers I} we pick a $ \mathcal{Q}_n \in \mathfrak{P}_{L_n}(v, S_v)$ that covers $ I-rv $.
We define $ L_{n+1}:=L_{n} \upharpoonright_{v}E^{+}(\mathcal{Q}_{n}) $.
Preservation of properties \labelcref{item: main 2}, \labelcref{item: main 3} and \labelcref{item: main 9}
follow directly from the construction.
Conditions \labelcref{item: main 1} and \labelcref{item: main 4} do not demand anything new for this step.
Properties~\labelcref{item: main 5} and \labelcref{item: main 6} are preserved by \cref{l:  no collapse of smallest EMsep}.
Since $ V_0=\varnothing $, property \labelcref{item: main 7} says nothing so far.
The definition of $ \mathcal{Q}_{n} $ and $ L_{n+1} $ used only $ L_{n}, v_n $ and choice function $ c $ as parameters all of which are in $ M_1 $, thus \labelcref{item: main 8} is preserved.

Assume now that $ \alpha>0 $.
In this case \labelcref{item: main 7} also has demands.
We choose $ \mathcal{Q}^{\omega\alpha+n} $ in accordance with \labelcref{item: main 7c} which is possible by combining \labelcref{item: main 7a} and \labelcref{item: main 7b} via \cref{cor: EMpathsys covers I}.
To fulfil \labelcref{item: main 7a} and \labelcref{item: main 7b} for $\omega \alpha +n +1$ we are going to pick an $ I $ according to the following \namecref{cl: desired I}. 
\begin{claim}\label{cl: desired I}
    There exists an $ I\in \mathcal{G}_{L_{\omega \alpha+n}}(v) $ such that:
\begin{enumerate} [label=(\Roman*)]
    \item\label{item: succ a} For every $ u\in V_{\alpha}-r $, $ S_u\setminus V_\alpha\in \mathfrak{S}_{L^{\omega \alpha+n}\upharpoonright_v I}(u) $;
    \item\label{item: succ b} $ L^{\omega \alpha+n}\upharpoonright_v I $ has the vertex-flame property for every $ u\in V_{\alpha}-r $;
    \item\label{item: succ c} $I\supseteq \bigcup_{k< n}\mathsf{in}_{Q^{\omega\alpha+k}}(v)\cup \bigcup_{k<n} \mathsf{in}_{Q_{\omega\alpha+k}}(v) =:\boldsymbol{F}$.
\end{enumerate}
\end{claim}
Suppose that we already know \cref{cl: desired I}.
By \cref{cor: EMpathsys covers I} we can pick a path-system $\mathcal{Q}_{\omega \alpha+n}\in \mathfrak{P}_{L_{\omega \alpha+n}}(v, S_v) $ that covers $ I-rv $.
We define $ L_{\omega \alpha+n+1}:=L_{\omega \alpha+n} \upharpoonright_{v}E^{+}(\mathcal{Q}_{\omega \alpha+n}) $. 
Conditions \labelcref{item: main 1,item: main 2,item: main 3,item: main 4,item: main 5,item: main 6} are preserved for the same reason as in the case $ \alpha=0 $.
Note that properties \labelcref{item: succ a} and \labelcref{item: succ b} of the desired set $ I $ are increasing in the sense that if they hold for some $ I $, then they remain true for every $I' \supseteq I $.
Indeed, the path-systems witnessing these properties for $ I $ also witness them with respect to $ I' $.
Conditions \labelcref{item: succ a} and \labelcref{item: succ b} guarantee \labelcref{item: main 7a} and \labelcref{item: main 7b} respectively for $\omega \alpha +n +1$. 
Preservation of \labelcref{item: main 9} is ensured by \labelcref{item: succ c}.
The definition of $ \mathcal{Q}_{\omega \alpha+n}, \mathcal{Q}^{\omega \alpha+n} $ and $ L_{\omega \alpha+n+1} $ rely only on the parameters $ L_{\omega\alpha+n},v_{\omega \alpha+n}, c$ and $ M_\alpha $ all of which are in $ M_{\alpha+1} $, thus we keep \labelcref{item: main 8} as well.

\begin{proof}[Proof of \cref{cl: desired I}]
Since $ F $ is finite ($ \left|F\right|\leq 2n $ follows directly from its definition) and $ L_{\omega \alpha+n} $ is a quasi-vertex-flame, we have $F\in \mathcal{G}_{L_{\omega \alpha+n}}(v) $.
We claim that it is possible to choose a witness $ \boldsymbol{\mathcal{Q}_{F}}=\{ Q_e: e\in F \} $ for $F\in \mathcal{G}_{L_{\omega \alpha+n}}(v) $ (where $ e $ is the last edge of $ Q_e $) in such a way that whenever a path in $ \mathcal{Q}_{F} $ leaves $ V_\alpha $ it never returns, in other words no path in $ \mathcal{Q}_F $ has an edge in $ \mathsf{in}_D(V_\alpha) $.
Indeed, suppose that $ \mathcal{Q}_F' $ is an arbitrary witnesses for $F\in \mathcal{G}_{L_{\omega \alpha+n}}(v) $ and let $ U_F $ be the set of the last common vertices of the paths in $ \mathcal{Q}_F' $ with $ V_\alpha $.
Then $ U_F-r $ is a finite subset of $ V_\alpha-r $ which is linked from $ r $ in $ D $. But then, since $ M_\alpha $ is an elementary submodel, $ U_F-r $ is linked from $ r $ in $ D_\alpha $ as well.
It follows from \cref{l: linked from r U} via the $ D_\alpha $-largeness of $ L_{\omega \alpha}\cap D_\alpha $ (see \cref{l: lim linked U from r}) and $ L_{\omega \alpha}\cap D_\alpha= L_{\omega \alpha+n}\cap D_\alpha $ that $ U_F-r $ remains linked from $ r $ in $ L_{\omega \alpha+n}\cap D_\alpha $.
This means that we can replace the initial segments of the paths in $ \mathcal{Q}_F' $ up to $ U_F $ in $ L_{\omega \alpha+n}\cap D_\alpha $ in such a way that these new initial segments have vertices only in $ V_\alpha $.
This modification of $ \mathcal{Q}_F' $ provides the desired $ \mathcal{Q}_F $.
 
We build an auxiliary digraph $ \boldsymbol{A} $ by adding a ``dummy'' vertex $ w_{u} $ for every $ u\in V_{\alpha}-r $ to $ L^{\omega \alpha+n} $ whose in-neighbours are $ S_u\setminus V_\alpha $ and has no out-neighbours.
Let $ \boldsymbol{W}:=(V_{\alpha}-r)\cup \{ w_u: u\in V_{\alpha}-r \} $.
Then $ A $ has the vertex-flame property at every $ w\in W $ by properties \labelcref{item: main 7a} and \labelcref{item: main 7b}, moreover, $ W $ is countable. We are going to choose $ I $ in such a way that $ A \upharpoonright_v I $ also has the vertex flame property for every $ w\in W $.
For the original digraph $ L^{\omega \alpha+n} $ this means that $ L^{\omega \alpha+n} \upharpoonright_v I $ has the vertex-flame property for every $ u\in V_{\alpha}-r $ (demanded by \labelcref{item: succ b}), furthermore, the preservation of the vertex-flame property for dummy vertex $ w_u $ ensures that $ S_u\setminus V_\alpha $ remains linked from $ r $ in $ L^{\omega \alpha+n} \upharpoonright_v I $ which can be thought of as ``half'' of condition \labelcref{item: succ a}.

By applying \cref{cor: I G-qusi} with $ A, v=v_{\omega \alpha+n} $ and $ W $, we obtain an $I^* \in \mathcal{G}_A(v)= \mathcal{G}_{L^{\omega \alpha+n}}(v)$ such that ${A \upharpoonright_v I^*}$ has the 
vertex-flame property for every $ w\in W $.
Let $ \mathcal{Q} $ be a system of internally disjoint $ r\rightarrow v $ paths in $ L_{\omega \alpha+n} $ such that 
\begin{enumerate}[label=(\roman*)]
    \item\label{item: i} $F\subseteq E^{+}(\mathcal{Q})\subseteq F\cup I^{*} $;
    \item\label{item: ii} $ I^{*}\setminus E^{+}(\mathcal{Q}) $  is finite;
    \item\label{item: iii} Whenever some $Q\in \mathcal{Q} $ is not a path in $ L^{\omega \alpha+n} $, then $ Q\in \mathcal{Q}_F $;
    \item\label{item: iv} $ I^{*}\setminus E^{+}(\mathcal{Q}) $ is minimal among path-systems satisfying the properties \labelcref{item: i,item: ii,item: iii}.
\end{enumerate}

We first show that $ \mathcal{Q} $ is well-defined:
if we take a path-system $ \mathcal{Q}_{I^{*}} $ witnessing $ I^{*}\in \mathcal{G}_{A}(v) $, then, since $ \mathcal{Q}_F $ is finite, there is a co-finite subset $ \mathcal{Q}_{I^{*}}' $ of $ \mathcal{Q}_{I^{*}} $ for which the path-system $\mathcal{Q}'_{I^{*}}\cup \mathcal{Q}_F $ is internally disjoint and hence satisfies \labelcref{item: i,item: ii,item: iii}.

We claim that $E^{+}(\mathcal{Q}) $ still has the property that $ I^{*} $ had, namely that ${A \upharpoonright_v E^{+}(\mathcal{Q}) }$ has the vertex-flame property for every $ w\in W $.
Suppose for a contradiction that $ A \upharpoonright_v E^{+}(\mathcal{Q}) $ does not have the vertex-flame property at some $ w\in W $.
Note that we necessarily must have $w \neq v$, because $\mcalQ$ witnesses $E^{+}(\mcalQ) \in \mcalG_{\upharpoonright_v}E^{+}(\mathcal{Q}) (v)$.
Let $ \mathcal{P}_w $ be a witness for $ \mathsf{in}_A(w)\in \mathcal{G}_{A \upharpoonright_v I^{*}}(w) $.
Then there is precisely one path $P \in \mathcal{P}_w $ that uses precisely one edge $ uv\in I^{*}\setminus E^{+}(\mathcal{Q}) $.
Thus $\mcalP_w$ witnesses $\mathsf{in}_{A \upharpoonright_{v} E^{+}(\mcalQ) + uv} (w) \in \mcalG_{A \upharpoonright_{v} E^{+}(\mcalQ) + uv} (w)$.
Note also that $ u\neq r $, moreover, $ rv\notin E(D) $ since otherwise $ rv\in E(A) $ and hence the initial segment $ P v $ of $P$ can be replaced by the single edge $ rv $ and this shows that $ \mathsf{in}_{A}(w)\in \mathcal{G}_{A \upharpoonright_v E^{+}(\mathcal{Q})}(w) $, which contradicts the choice of $ w $.
Thus we may apply \cref{cor: G-quasi add one} with $ A \upharpoonright_v E^{+}(\mathcal{Q})+uv $, $ w $ and $ uv $ and we obtain a vertex set $ S\ni v $ which is linked from $ r $ in $ A \upharpoonright_v E^{+}(\mathcal{Q})+uv $ by a path-system~$\mathcal{P}_S$, such that~$S$ separates $ N_{A \upharpoonright_v E^{+}(\mathcal{Q})+uv}^{-}(v)-u $ from~$r$.
In particular,~$uv$ is the last edge of some~${P_{uv} \in \mathcal{P}_S}$.
We can assume $ S\cap V_\alpha=\emptyset $ by taking $ S:=S\setminus V_\alpha $ instead, because the vertices in $ V_{\alpha}-r $ do not have outgoing edges in $ A $ and hence it is still a separator.
We modify $ \mathcal{Q} $ in the following way.
Whenever $ Q\in \mathcal{Q} $ does not meet $ S-v $, then we let $ Q':=Q $.
Note that since $S$ separates $N_{A \upharpoonright_{v} E^{+}(\mcalQ)+uv}^{-}(v)-u$ from $r$ and no path of $\mcalQ$ uses one of the edges in $\mathsf{in}_{A}(v)\setminus E^{+}(\mcalQ)$ by definition, such a path is not a path in $L^{\omega \alpha +n}$, thus by \labelcref{item: iii} it is a path of $\mcalQ_{F}$.
Any other path $ Q\in \mathcal{Q} \setminus \mcalQ_{F} $ meets $ S-v $.
In this case we take the last common vertex $ v_Q $ of $ Q $ with $ S-v $ and replace the initial segment $ Q v_Q $ by the unique path in $ \mathcal{P}_S $ that terminates at $ v_Q $ to obtain $ Q' $.
Since the paths in $\mcalP_{S}$ are paths in $A \upharpoonright_{v} E^{+}(\mcalQ)$, the paths are disjoint from $Q \in \mcalQ_{F}$ by the construction of $A$.
Likewise, no path of $\mcalP_{S}$ can share a vertex with one of the segments $v_{Q}Q$ other than $v_{Q}$, since their union then would contain an $r$--$v$-path in $A \upharpoonright_{v} E^{+}(\mcalQ)+uv$ avoiding $S-v$.
Thus the constructed paths are internally disjoint and $\mathcal{Q}' :=\braces{ Q' \colon Q\in \mathcal{Q} }\cup \{ P_{uv} \} $ satisfies \labelcref{item: i,item: ii,item: iii} and witnesses via $ uv $ that $ \mathcal{Q} $ does not satisfy \labelcref{item: iv}, a contradiction.

Choosing $ I $ to be $ E^{+}(\mathcal{Q}) $ is ``almost'' suitable. Indeed, properties \labelcref{item: succ b} and \labelcref{item: succ c} would be satisfied as well as ``half'' of \labelcref{item: succ a}.We shall define $ I $ as a superset of $ E^{+}(\mathcal{Q}) $ guaranteeing the ``other half'' of \labelcref{item: succ a}, namely that $ S_u\setminus V_\alpha$ remains linked to $ u $ in $L^{\omega \alpha+n}\upharpoonright_v I $ for every $ u\in V_{\alpha}-r $.
Note that for $ T:=T_{L^{\omega \alpha+n},v} $ (see the definition at the end of \cref{subsection EM sep path}) we have $ T\cap V_\alpha=\varnothing $ because in $ L^{\omega \alpha +n} $ the vertices in $ V_\alpha-r $ have no outgoing edges and $ T $ is a minimal separation.
We are going to choose $ I $ in such a way that $ T$ remains linked to $ v $ in $L^{\omega \alpha +n}\upharpoonright_v I$.
By \cref{l: largest EMsep} this ensures that $ S_u\setminus V_\alpha$ remains linked to $ u $ in $ L^{\omega \alpha+n}\upharpoonright_v I $ for every $ u\in V_{\alpha}-r $. Let $ \mathcal{P}\in \mathfrak{P}_{L^{\omega \alpha+n}}(v, T) $ and let $ \mathcal{P}' $ consists of the terminal segments of the paths in $ \mathcal{P} $ from $ T $.
For $ Q\in \mathcal{Q} $, let $ Q' $ be the terminal segment of $ Q $ from the last common vertex with $T\cup V_\alpha $.
\Cref{cor: Pym+} applied in digraph $ L_{\omega \alpha+n} $ with vertex set $ T\cup V_\alpha $ and 
vertex $ v $ together with path-systems $ \mathcal{P}' $ and $ \mathcal{Q}' $ provides a system $ \mathcal{R}' $ (see \cref{fig R}) of $ T\cup V_\alpha\rightarrow v $ paths such that $ V(R_0)\cap V(R_1)\subseteq\{ r,v \} $ for every distinct $ R_0, R_1\in \mathcal{R}' $,
\[
    \boldsymbol{I}:=E^{+}(\mathcal{R}') \supseteq E^{+}(\mathcal{Q}') =E^{+}(\mathcal{Q})
\] 
and for every $ t\in T $ there is a unique $ R_t\in \mathcal{R}' $ with first vertex $ t $.
We claim that paths $ R_t $ lie completely in the subdigraph $ L^{\omega \alpha+n} $ of $ L_{\omega \alpha+n}$.
This is true by definition for $ R_t\in \mathcal{P}'$.
If it is not the case, then $ R_t $ consists of the initial segment of some $ P\in \mathcal{P}' $ up to some $ v_R $ and the terminal segment of some $ Q\in \mathcal{Q}' $ from $ v_R $.
If $ Q $ itself lies in $ L^{\omega \alpha+n} $ then we are done again.
If it is not the case, then $ Q $ is a terminal segment of a path in $ \mathcal{Q}_F $ see \labelcref{item: iii}.
Clearly $ v_R\notin V_\alpha $ because no path in $ \mathcal{P}' $ meets $ V_\alpha $.
But then the terminal segment of $ Q $ from $ v_R $ lies entirely in $ L^{\omega \alpha+n} $ because the paths in $ \mathcal{Q}_F $ never returns to $ V_\alpha $ once they left it.
Therefore $ R_t $ lies in $ L^{\omega \alpha+n} $ in all possible cases.
Thus the paths $ \{ R_t: t\in T \} $ link $ T $ to $ v $ in $ L^{\omega \alpha+n} $.

We extend the paths in $ \mathcal{R}' $ backwards to obtain a path-system $ \mathcal{R} $ witnessing $ I\in \mathcal{G}_{L_{\omega \alpha+n}}(v) $.
For $t\in T $, we take the initial segment of the unique $ P_t\in \mathcal{P} $ through $ t $ until $ t $.
These extended paths meet $ V_\alpha $ only at $ r $ because in $ L^{\omega \alpha +n} $ the vertices in $ V_\alpha-r $ have no outgoing edges.
There are only finitely many paths $ R $ in $ \mathcal{R}' $ whose first vertex is in $ V_\alpha $, each of which is a terminal segment of a path $ Q_R\in \mathcal{Q} $.
By  property \labelcref{item: iii} these $ Q_R $ are in $ \mathcal{Q}_F $.
As a backward extension of such an $ R $ we choose simply $ Q_R $ itself.
The new initial segments in this case have vertices only in $ V_\alpha $ and therefore meet the initial segments added to paths $ R_t $ only at $ r $.
Thus the resulting $ \mathcal{R} $ is really internally disjoint which completes the proof of \cref{cl: desired I}.
\end{proof}

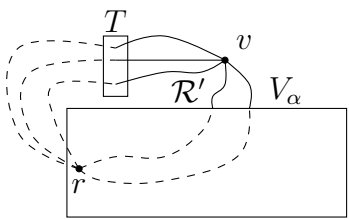
\begin{figure}[H]
    \centering
    \begin{tikzpicture}[scale=0.8]
        \draw  (-2.4,-1) rectangle (2.2,-2.8);
        \draw  (-1.8,0.2) rectangle (-1.4,-0.8);
        \node (v2) at (0.2,-0.2) {};
        \node at (-1.6,0.45) {$T$};
        \node (v1) at (-2.2,-2) {};
        \node at (1.2,-0.7) {$V_\alpha$};
        \node[red] at (-0.4,-0.7) {$\mathcal{R}'$};
        
        \draw[red]  plot[smooth, tension=.7] coordinates {(-1.6,0) (-0.8,0.2) (v2)};
        \draw[red]  plot[smooth, tension=.7] coordinates {(-1.6,-0.2) (-0.8,-0.2) (0.2,-0.2)};
        \draw[red]  plot[smooth, tension=.7] coordinates {(-1.6,-0.6) (-0.8,-0.4) (-0.2,-0.4) (0.2,-0.2)};
        
        \draw[red,dashed]  plot[smooth, tension=.7] coordinates {(-2.2,-2) (-3.2,-1.4) (-3.2,0) (-1.6,0)};
        \draw[red,dashed]  plot[smooth, tension=.7] coordinates {(-2.2,-2) (-3,-1.4) (-3,-0.4) (-1.6,-0.2)};
        \draw[red,dashed]   plot[smooth, tension=.7] coordinates {(-2.2,-2) (-2.6,-1.2) (-2.6,-0.6) (-1.6,-0.6)};
        
        \draw[red,dashed]   plot[smooth, tension=.7] coordinates {(-2.2,-2) (-1.8,-2.2) (-0.8,-2.2) (0.4,-1.8) (0.6,-1)};
        \draw[red,dashed]  plot[smooth, tension=.7] coordinates {(-2.2,-2) (-1.6,-1.8) (-0.8,-1.8) (-0.2,-1.4) (0,-1)};
        
        \draw[red]  plot[smooth, tension=.7] coordinates {(0,-1) (0,-0.8) (0.2,-0.6) (0.2,-0.2)};
        \draw[red]  plot[smooth, tension=.7] coordinates {(0.6,-1) (0.6,-0.6) (0.2,-0.2)};
        \node[below] at (v1) {$r$};
        \node[above right] at (v2) {$v$};
        \foreach \x in {v1,v2} \draw[fill] (\x) circle [radius=.05];
    \end{tikzpicture}
    \caption{The path-system $ \mathcal{R} $.}\label{fig R}
\end{figure}

\end{document}